\documentclass[10pt]{article}
\usepackage{geometry}                % See geometry.pdf to learn the layout options. There are lots.
\usepackage{graphicx}
\usepackage{fullpage}
\usepackage{amsmath}
\usepackage{amssymb}
\usepackage{amsthm}
\usepackage{url}
\usepackage{epsfig}
\usepackage{color}
\usepackage{refcount}
\usepackage{verbatim}
\usepackage{enumerate}
\usepackage{multirow}
\usepackage{framed}
\usepackage{float}
\usepackage{slashbox}
\usepackage{caption}
\usepackage{subcaption}
\usepackage{array}
\usepackage[normalem]{ulem}

\usepackage[square,numbers,sort&compress]{natbib}
\usepackage[title]{appendix}

\usepackage{xr}
\usepackage[breaklinks=true,colorlinks,citecolor=blue,linkcolor=blue]{hyperref}

\usepackage{calc}

\usepackage{tikz}
\usepackage{xifthen}
\usetikzlibrary{calc}
\usepackage{epstopdf}
\usetikzlibrary{positioning}
\usepackage{pgffor}
\usepackage[ruled,vlined]{algorithm2e}

\usepackage{color}
\definecolor{clemson-orange}{RGB}{234,106,32}
\definecolor{chicago-maroon}{RGB}{128,0,0}
\definecolor{cincinnati-red}{RGB}{190,0,0}
\definecolor{soft-cyan}{RGB}{68,85,90}
\usepackage{fullpage}
\usepackage{multicol}

\usepackage[utf8]{inputenc}
\usepackage{array}
\newcolumntype{L}[1]{>{\raggedright\let\newline\\\arraybackslash\hspace{0pt}}m{#1}}
\newcolumntype{C}[1]{>{\centering\let\newline\\\arraybackslash\hspace{0pt}}m{#1}}
\newcolumntype{R}[1]{>{\raggedleft\let\newline\\\arraybackslash\hspace{0pt}}m{#1}}

\theoremstyle{definition}
\newtheorem{theorem}{Theorem}[section]
\newtheorem{lemma}[theorem]{Lemma}

\newtheorem{definition}[theorem]{Definition}
\newtheorem{remark}[theorem]{Remark}
\newtheorem{assumption}{Assumption}

\newtheorem{alg}{Algorithm}

\usepackage{lineno}
\makeatletter
\newcommand*\patchAmsMathEnvironmentForLineno[1]{%
  \expandafter\let\csname old#1\expandafter\endcsname\csname #1\endcsname
  \expandafter\let\csname oldend#1\expandafter\endcsname\csname end#1\endcsname
  \renewenvironment{#1}%
     {\linenomath\csname old#1\endcsname}%
     {\csname oldend#1\endcsname\endlinenomath}}% 
\newcommand*\patchBothAmsMathEnvironmentsForLineno[1]{%
  \patchAmsMathEnvironmentForLineno{#1}%
  \patchAmsMathEnvironmentForLineno{#1*}}%
\AtBeginDocument{%
\patchBothAmsMathEnvironmentsForLineno{equation}%
\patchBothAmsMathEnvironmentsForLineno{align}%
\patchBothAmsMathEnvironmentsForLineno{flalign}%
\patchBothAmsMathEnvironmentsForLineno{alignat}%
\patchBothAmsMathEnvironmentsForLineno{gather}%
\patchBothAmsMathEnvironmentsForLineno{multline}%
}

\newcommand{\E}{\mathbb{E}}

\def\itemrange#1{%
\addtocounter{enumi}{1}%
\edef\labelenumi{\theenumi--\noexpand\theenumi.}%
\addtocounter{enumi}{-1}%
\addtocounter{enumi}{#1}%
\item
\def\labelenumi{\theenumi.}}

\numberwithin{equation}{section}  % If you number theorems, etc. within sections,
\title{Approximation Algorithm for Generalized Budgeted Assignment Problems and Applications in Transportation Systems}
\author{
    Hongyi Jiang\textsuperscript{a}, Samitha Samaranayake\textsuperscript{b}
}
\date{
    \textsuperscript{a}Department of Systems Engineering, City University of Hong Kong\\
    \textsuperscript{b}School of Civil and Environmental Engineering, Cornell University, USA\\
    \texttt{hongyi.jiang@cityu.edu.hk}, \texttt{samitha@cornell.edu} \\
    \today
}

\begin{document}
\maketitle

\begin{abstract}
Motivated by a transit line planning problem in transportation systems, we investigate the following capacitated assignment problem under a budget constraint. Our model involves $L$ bins and $P$ items. Each bin $l$ has a utilization cost $c_l$ and an $n_l$-dimensional capacity vector. Each item $p$ has an $n_l$-dimensional binary weight vector $r_{lp}$, where the $1$s in $r_{lp}$ (if any) appear in consecutive positions, and its assignment to bin $l$ yields a reward $v_{lp}$. The objective is to maximize total rewards through an assignment that satisfies three constraints: (i) the total weights of assigned items do not violate any bin's capacity; (ii) each item is assigned to at most one open bin; and (iii) the overall utilization costs remain within a total budget $B$.

We propose the first randomized rounding algorithm with a constant approximation ratio for this problem. We then apply our framework to the motivating transit line planning problem, presenting corresponding models and conducting numerical experiments using real-world data. Our results demonstrate significant improvements over previous approaches in addressing this critical transportation challenge.
 %Furthermore,  we show that our framework is more general---it is able to solve a larger class of budgeted assignment problems within transportation systems (e.g., generalized team orienteering problem) and beyond.

 ~
 
\end{abstract}

\section{Introduction}
This paper presents a framework for addressing a critical challenge in transportation systems: optimizing the selection of high-capacity fixed-route transit lines to maximize the total reward from passenger route coverage, subject to a specified operational budget constraint. This problem arises from urban transit systems' pressures of satisfying passenger demand with inadequate resource allocation, necessitating innovative solutions for effective line planning. Our approach aims to enhance the overall performance and efficiency of public transportation networks through strategic line planning. The primary challenge lies in developing an algorithm that is both theoretically sound and practically efficient, capable of deriving high-quality solutions for large-scale problems, such as those encountered in complex urban environments like New York City. This research makes contributions in both discrete optimization and transportation planning, proposing a novel algorithmic framework that bridges the gap between theoretical guarantees in algorithm design and practical applicability in real-world transit planning scenarios.

To address this challenge, we study the transportation problem in the context of a new assignment problem, which we call the Generalized Budgeted Assignment Problem (GBAP). Prior to introducing the formal definition, we present a warm-up example related to the aforementioned transportation line planning problems. This simple example aims to provide context and facilitate understanding of the subsequent formal definitions of the GBAP.

For ease of explanation, this example is relatively restrictive compared to real-world applications. Thus, it is important to note that we will study a more generalized line planning problem in Section \ref{sec:applications}. This extended discussion will still be based on the GBAP and corresponding approximation algorithm we develop, and will utilize real-world data.

\subsection{Warm-up Example}
This example involves selecting a set of fixed-route transit lines to operate such that passenger coverage is maximized, while adhering to vehicle capacity constraints and satisfying a budget constraint. By framing this as an assignment-type integer program, we aim to find the optimal solution that balances cost and service efficiency. More specifically, 
consider a transit system with four stops: $A$, $B$, $C$, and $D$. There are three candidate transit lines: line $1$, which travels from $A$ to $C$ to $D$; line $2$, which travels from $A$ to $B$ to $C$ to $D$; and line $3$, which travels from $C$ to $A$ to $B$ to $D$. These lines cost $20$, $40$, and $30$ respectively to operate. Each transit line operates a bus with a capacity of $2$ passengers. This capacity applies uniformly at all points along the route. We have six passengers: passengers $1$ and $2$ need to travel from $C$ to $B$, passengers $3$, $4$, and $5$ from $C$ to $D$, and passenger $6$ from $A$ to $B$. These transit lines, their costs, and passenger routes are illustrated in Figure \ref{fig:warmup} and Figure \ref{fig:warmup-2}.

% \begin{figure}
%     \centering
%     \includegraphics[width=0.48\linewidth]{DAM-warmup-exp.png}
%     \caption{Warm-up Example}
%     \label{fig:warmup}
% \end{figure}

\begin{figure}[h]
    \centering
    \includegraphics[width=0.48\linewidth]{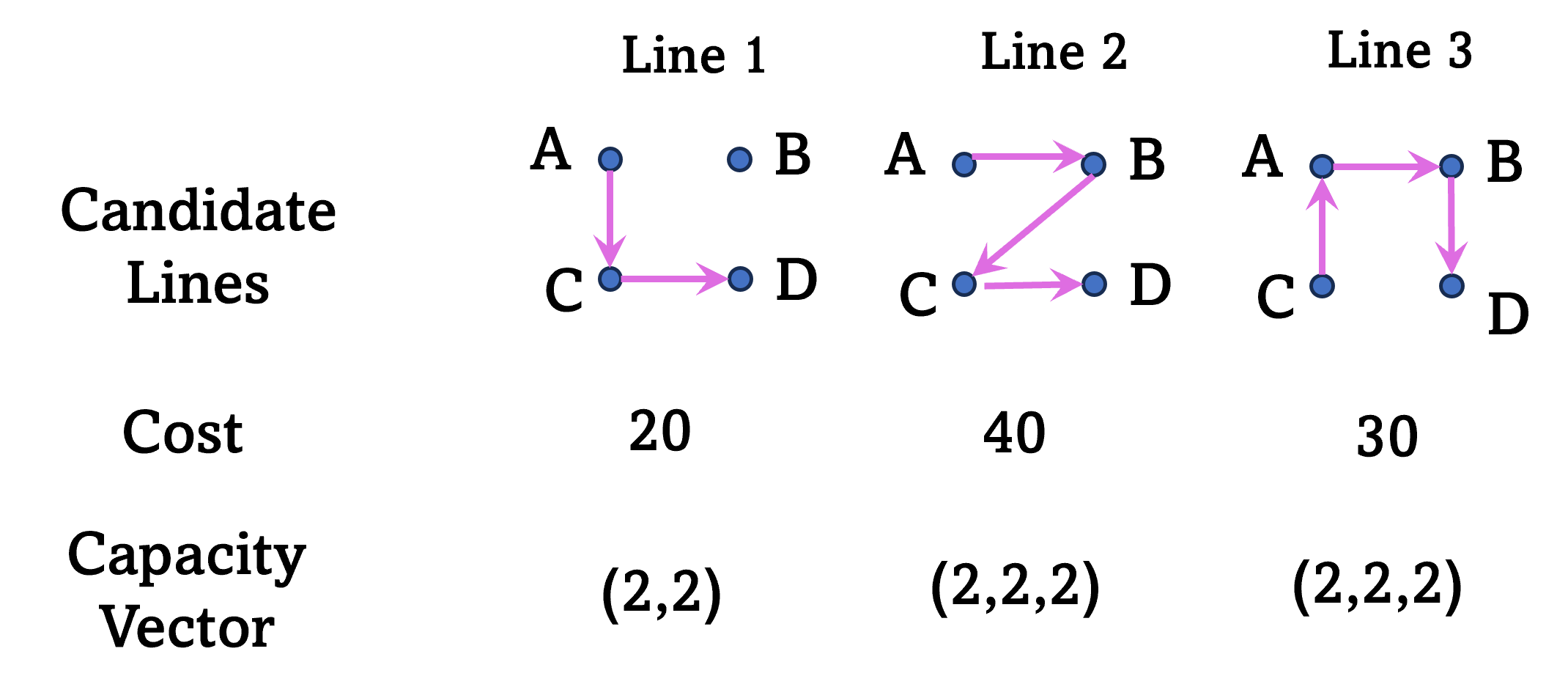}
    \caption{This figure displays the three candidate transit lines (Line $1$, Line $2$, and Line $3$) with their routes, operational costs, and capacity vectors. The routes are shown as directed graphs connecting stations $A$, $B$, $C$, and $D$. The cost of operating each line is given, along with its capacity vector. The capacity vector indicates the maximum number of passengers that can be accommodated on each segment of the line.}
    \label{fig:warmup}
\end{figure}

\begin{figure}[h]
    \centering
    \includegraphics[width=0.9\linewidth]{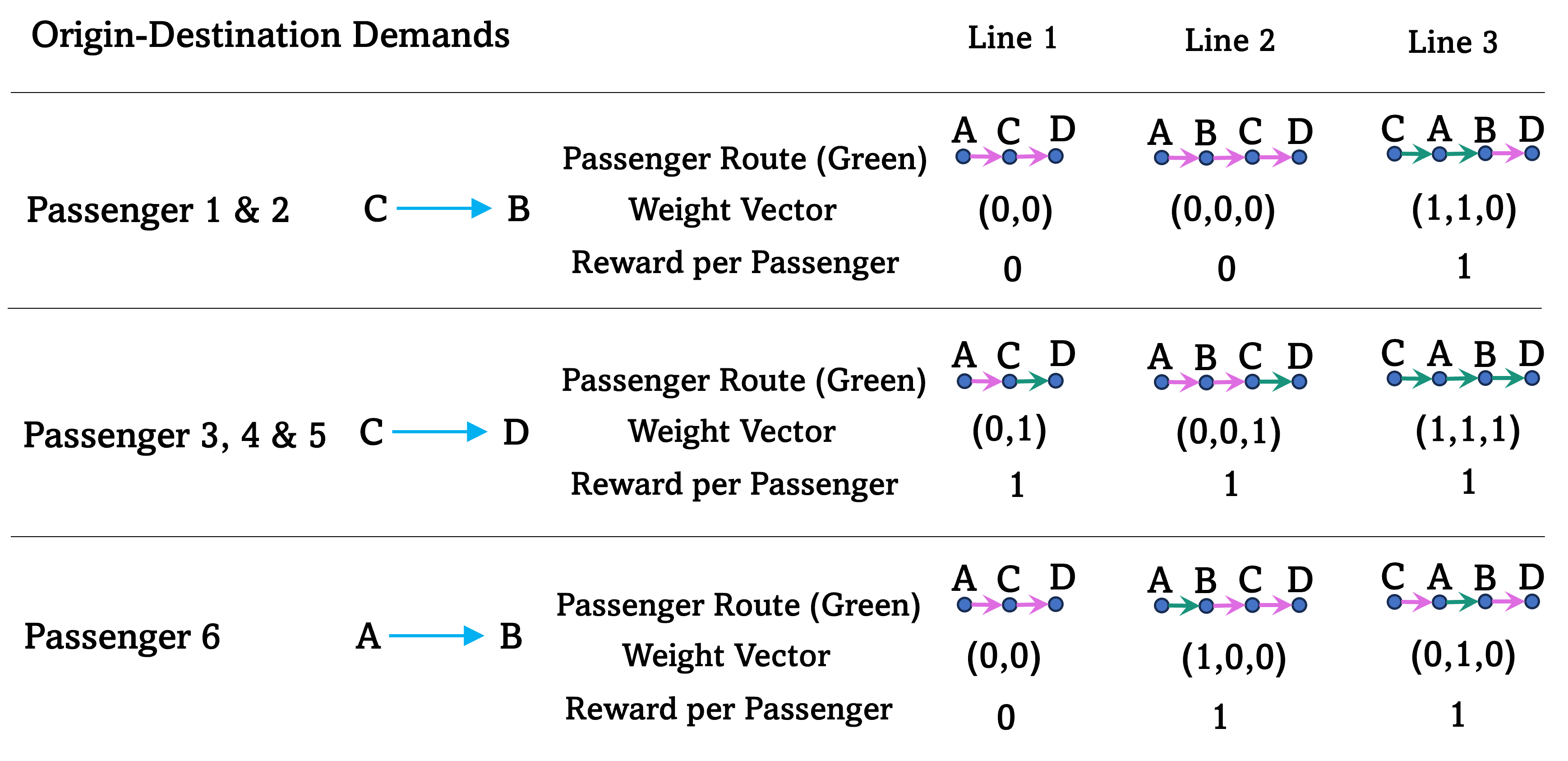}
    \caption{The figure shows the origin-destination pairs for passengers $1-6$, the routes they take on each transit line (with green arrows indicating used segments), the weight vectors for each line (where $1$ represents a used segment and $0$ an unused segment), and the reward for assigning each passenger to the line, where $1$ means the line fully covers the passenger’s origin-destination pair and $0$ means it does not.}
    \label{fig:warmup-2}
\end{figure}
For simplicity, we assume each passenger must travel directly from their origin to their destination on a single bus line (i.e., no transfers). To serve a passenger, the line chosen must be open and cover the passenger's entire route. The total cost of opened lines must stay within the given budget of $70$. Therefore, we can open at most $2$ lines.  Our goal is to choose a set of lines to open within the budget and assign passengers to these lines, maximizing the number of passengers served while adhering to these rules. If a line covers a passenger's route, assigning this passenger to the line generates a reward of $1$; otherwise, the reward is $0$. The objective is to maximize the overall generated reward in this context. 

 To solve this problem, we can construct a capacity vector for each of the three lines, as shown in Figure \ref{fig:warmup}. Line $1$ has a capacity vector of $(2,2)$, while $2$ and $3$ have capacity vectors of $(2,2,2)$. Each entry represents the maximum number of passengers that can be accommodated on a segment of the line. The length of each vector corresponds to the number of edges in that line.
Figure \ref{fig:warmup-2} illustrates how passengers utilize different lines and the corresponding weight vectors. For example, passengers $1$ and $2$, traveling from $C$ to $B$, have weight vectors $(0,0)$ for line $1$, $(0,0,0)$ for line 2, and $(1,1,0)$ for line $3$. These vectors indicate which segments of each line are used by the passengers, with $1$ representing a used segment and $0$ an unused segment.
The reward of assigning passengers to a line depends on whether the line covers their route. For instance, assigning passengers $1$ and $2$ to line $1$ or $2$ yields \emph{no reward} as these lines do not cover their route, while assigning them to line $3$ gives a reward of $1$ per person.
To obey capacity constraints, the sum of weight vectors for passengers assigned to a line cannot exceed that line's capacity vector entry-wise. For example, we cannot assign passengers $1$, $2$, and $3$ to line $3$ as it would exceed the capacity $(3 > 2)$ for the edges from $C$ to $A$ and $A$ to $B$.

\begin{remark}
    When a line covers a passenger's route, the weight vector contains 1s appearing in adjacent positions, reflecting that passenger routes consist of consecutive edges on the lines. For lines not covering a passenger's route, the weight vectors are defined as all zeros for consistency, though they could be arbitrary since the reward for these assignments is $0$.
\end{remark}

Based on these definitions, we can formulate this problem as an assignment-type integer program. This formulation is a specific case of GBAP, which we will present in details in Section \ref{sec:def-gbap}, specifically in (\ref{eq:2D-IP-BAP-UW}). It can be verified directly that in this example, the optimal solution is to operate lines $2$ and $3$, covering 5 passengers, with only one passenger (from C to D) left unserved.

\subsection{Definition of  Generalized Budgeted Assignment Problem}\label{sec:def-gbap}

Building on the warm-up example, we now formally define GBAP. This problem involves selecting a set of bins within a given budget and assigning a set of items to these selected bins under multidimensional capacity constraints, with the objective of maximizing the assignment reward.  To help understand this model, one can think of transit lines as bins and passengers as items to be allocated. Specifically, consider a system with $L$ bins and $P$ items. Each bin $l \in [L]$\footnote{Here, $[n]$ denotes the set $\{1,2,\ldots,n\}$ for any integer $n$.} is characterized by a utilization cost $c_l \in \mathbb{Q}_+$ and an $n_l$-dimensional capacity vector $f_l = \left(f_{l}^{(1)}, \ldots, f_{l}^{(n_l)}\right) \in \mathbb{Z}^{n_l}_+$. In our warm-up example, the capacity vector $f_l$ has uniform entries across all dimensions. However, in our generalized model, this capacity can vary across dimensions (see Remark \ref{rmk:def} for further details). Additionally, the dimension $n_l$ corresponds to the number of edges involved in a line in the warm-up example.
 Each item $p \in [P]$ has an $n_l$-dimensional weight vector for each bin $l$: $r_{lp} = \left(r_{lp}^{(1)}, \ldots, r_{lp}^{(n_l)}\right) \in \{0,1\}^{n_l}$ , where the $1$s in $r_{lp}$ (if any) appear in consecutive positions within the vector. This structure reflects the assumption that routes consist of consecutive edges, as motivated by the warm-up example where routes follow contiguous edges in a line.  Assigning item $p$ to bin $l$ yields a reward $v_{lp} \geq 0$. The objective is to maximize the total reward through an assignment that satisfies three constraints: (i) the total weight of items assigned to each bin does not exceed the bin's capacity in any dimension, (ii) each item is assigned to at most one bin, (iii) items can only be assigned to utilized bins, and (iv)  the total cost of utilized bins does not exceed a given budget $B \in \mathbb{Q}_+$. 

The GBAP problem can be formulated as a $0$-$1$ linear integer programming problem:
\begin{subequations}\label{eq:2D-IP-BAP-UW}
\begin{align}
    \max_{y,x}~~&\sum_{p\in[P]}\sum_{l\in[L]}v_{lp}x_{lp}\\
    \mbox{s.t.}~~&\sum_{l\in [L]}c_ly_l\leq B\label{eq:original-IP-budget1}\\
    &\sum_{p\in [P]}x_{lp}\cdot r^{(i)}_{lp}\leq f_{l}^{(i)}\cdot y_l~~~\forall l\in[L], i\in [n_l]\label{eq:original-IP-capacity}\\
    &\sum_{l\in[L]}x_{lp}\leq 1~~~\forall p\in [P]\label{eq:original-IP-rho}\\
    &x_{lp}\in\{0,1\}~~~\forall p\in[P],~l\in[L]\label{eq:original-IP-x01}\\
    &y_l\in\{0,1\}~~~\forall l\in[L].
\end{align}
\end{subequations}
(\ref{eq:original-IP-budget1}) ensures that the budget constraint is respected. (\ref{eq:original-IP-capacity}) guarantees that the multidimensional capacity of bins is not breached by the assigned items.  (\ref{eq:original-IP-rho}) requires that item $p$ can be assigned at most one bin. If item $p$ is assigned to bin $l$, then $x_{lp}=1$. Otherwise, $x_{lp}=0$.  (\ref{eq:original-IP-x01}) guarantees that each bin can have at most one copy of each item. Similarly, $y_l=1$ indicates that bin $l$ is utilized. Items can only be matched to utilized bins.

\begin{remark}\label{rmk:def}
While the warm-up example uses binary rewards ($0$ or $1$), our model allows for any non-negative (rational) rewards depending on different passengers (items) and transit lines (bins). This flexibility enables the study of more complex transportation scenarios, such as maximizing social welfare defined according to passenger characteristics like income and age. For further details and discussions on applications, please refer to Section \ref{sec:applications}.

Additionally, while the capacity vector in our warm-up example contains uniform values, our model allows for different values across various dimensions. This feature provides enhanced flexibility, making the model applicable to a broader range of scenarios.
\end{remark}

For the rest of this paper, we define $k$ as the ratio of the budget to the largest line cost:
\begin{align}\label{eq:k-def}
    k := \frac{B}{\max_{l\in[L]}c_{l}}
\end{align}

\noindent This notation will be used frequently throughout our subsequent discussion and analysis.

\subsection{Linear Programming Relaxation of Integer Program \eqref{eq:2D-IP-BAP-UW}}

We will develop a randomized rounding algorithm based on the solution to the linear programming (LP) relaxation of \eqref{eq:2D-IP-BAP-UW}. While a natural approach to relaxing \eqref{eq:2D-IP-BAP-UW} would be to directly set $x_{lp}$ and $y_l$ as continuous variables, this method yields a loose solution, as noted in \cite{fleischer2011tight} and \cite{perivier2021real}, thereby limiting the performance of LP-based randomized rounding algorithms. Consequently, we will employ the LP relaxation \eqref{eq:2d-IP-config-original} initially proposed by \cite{fleischer2011tight}, augmented with an additional budget constraint.

For each bin $l\in [L]$, define 
\begin{align}\label{eq:Il}
    I_l:=\left\{S\subseteq [P]: \sum_{p\in S} r^{(i)}_{lp}\leq f_{l}^{(i)}~~~\forall i\in [n_l] \right\}
\end{align}
\noindent In other words, $I_l$ is the family of all nonempty feasible assignments of items to $l$, where a feasible assignment is a set $S\subset [P]$ such that bin $l$ has enough capacity to contain all the items in $S$, {and $\emptyset\notin I_l$}. For instance, in the warm-up example, if $l$ refers to line $3$, then $S:=\{\text{passenger } 1, 2, 3\} \notin \mathcal{I}_l$ because it violates the capacity constraints, whereas $S:=\{\text{passenger } 1, 2\} \in \mathcal{I}_l$.

Then the LP relaxation of (\ref{eq:2D-IP-BAP-UW}) can be stated as follows:
\begin{subequations}\label{eq:2d-IP-config-original}
\begin{align}
    \underset{\{X_{lS}\}}{\max}~~~~~~~ &\sum_{p\in [P]}\sum_{l\in[L]}\sum_{S\in I_l:p\in S}v_{lp}X_{lS}\\
    s.t. ~~~~~~~~&\sum_{S\in I_l}X_{lS}\leq 1~~~~\forall l\in [L]\label{eq:one-assignment}\\\
    &\sum_{l\in[L]}\sum_{S\in I_l:p\in S}X_{lS}\leq 1~~~~\forall p\in [P]\label{eq:config-rho}\\
    &\sum_{l\in[L]}\sum_{S\in I_l}c_{l}X_{lS}\leq B\label{eq:config-budget1}\\
    &X_{lS}\in [0,1]\label{eq:config-original-integer}
\end{align}
\end{subequations}
When $X_{lS}$ is restricted to $\{0,1\}$, (\ref{eq:one-assignment}) enforces that each bin can be matched with at most one feasible assignment.  Constraints (\ref{eq:config-rho}) and (\ref{eq:config-budget1}) correspond to (\ref{eq:original-IP-rho}) and (\ref{eq:original-IP-budget1}) respectively. This LP relaxation can be solved in polynomial time {using the} ellipsoid method as in \cite{fleischer2011tight,perivier2021real}. 

\medskip

\noindent {\bf Paper Structure:} The remainder of this paper is organized as follows:
Section \ref{sec:review-and-contribution} presents related work, challenges, and our contributions.
Section \ref{sec:overview} provides a high-level overview of our approach.
Sections \ref{sec:prelim} and \ref{sec:alg-1d} detail the technical foundations, algorithm, and its analysis.
Section \ref{sec:applications} presents numerical experiments using real-world data.
Section \ref{sec:conclusion} concludes the paper and discusses future research directions.

\section{Related Work, Challenges and Contributions}\label{sec:review-and-contribution}

\subsection{Related Work}
 
The GBAP is a variant of separable assignment problem (SAP). In SAP, there are $L$ bins and $P$ items. For each item $p$ and bin $l$, there
is a reward $v_{lp}$ generated when item $p$ is assigned to
bin $l$. Moreover, for each bin $l$, there is a separate
packing constraint, which means that only certain subsets of
the items can be assigned to bin $l$ when the subset satisfies the constraint. The goal is to find an assignment of items to the bins such that
all the sets of items assigned to bins do not violate packing constraints, each item is
assigned to at most one bin, and the total reward of assignments is optimal. Fleisher et al. \cite{fleischer2011tight} propose a
$\beta\left(1-\frac{1}{e}\right)$-approximation for SAP assuming
that there is a $\beta$-approximation
algorithm for the single bin subproblem. Moreover, they
showed that a special case of SAP called the capacitated distributed caching problem (CapDC) cannot be approximated
in polynomial time with an approximation factor better
than $\left(1-\frac{1}{e}\right)$ unless NP$\subseteq$ DTIME$(n^{O(\log \log n)})$. In CapDC, there are $L$ cache locations (bins), $P$ requests (items) and $T$ different files. Bin $l$ has capacity $A_l$, and file $t$ has size $a_t$. Request $p$ is associated with some file $t_p$ and a value $R_p$. If a request $p$ is connected to cache location $l$, then there is a cost $c_{lp}$ and the reward is $R_p-c_{lp}$. The total size of the files corresponding to connected requests should be no larger than the capacity of the bin. %\SSC{I'm assuming this section is not too similar to what's in the smart transit paper.}

The maximum generalized assignment problem (GAP) is an important special case of SAP, in which each bin $l$ has size $A_l$, each item $p$ has size $a_{lp}$ in each bin $l$ and a feasible assignment  means that the total size of assigned items for each bin is not larger than the bin's capacity. GAP is proved to be APX-
hard by Chekuri et al. \cite{chekuri2005polynomial}. The best known approximation algorithm for
GAP is from \cite{feige2006approximation} and achieves an
approximation factor of $\left(1-\frac{1}{e}+\delta\right)$ for a small constant
$\delta>0$. 

Calinescu et al. \cite{calinescu2011maximizing} also give a $\left(\beta(1- \frac{1}{e})\right)$-approximation for SAP. Moreover, they give a $\left(1-\frac{1}{e}-\epsilon\right)$-approximation algorithm for GAP with the additional restriction {of a budget constraint that} on the number of bins that can be used. \cite{kulik2013approximations} propose a $\left(1-\frac{1}{e}-\epsilon\right)$-approximation method for GAP with knapsack constraints, where there is a multi-dimensional budget vector and each item has a multi-dimensional cost vector same for each bin. The total cost vector of assigned items in a feasible assignment is no larger than the budget vector entry-wise.

A generalization of SAP, which is called as $\rho$-SAP, is studied by Bender et al. \cite{bender2015packing} (it is called $k$-SAP by \cite{bender2015packing}, where $k$ is the number of allowed bins there, but here we make it consistent with our notation to avoid confusion). In $\rho$-SAP each item can be assigned to at most $\rho$ different
bins, but at most once to each bin. The goal is to find an assignment of items to bins that maximizes total reward. They present a $\beta\left(1- \frac{1}{e^{\rho}}\right)$-approximation algorithm for $\rho$-SAP under the assumption that there is a $\beta$-approximation algorithm for the single bin subproblem. If the single bin subproblem admits an FPTAS, then the approximation factor is $\left(1-\frac{1}{e^{\rho}}\right)$. The scheme can also be extended to the different upper bounds of number of allowed bins for different items: if the upper bound is $\rho_i$ for item $i$, and $\rho=\min_i \rho_i$, then their algorithm yields $\beta\left(1- \frac{1}{e^{\rho}}\right)$ approximation. 

None of the aforementioned studies are applicable to scenarios that include budget constraints on bins.

In P\'erivier et al.'s study \cite{perivier2021real}, the focus is on the Real-Time Line Planning Problem (RLPP), which can be modeled as GBAP. In this problem setting, each passenger is restricted to selecting a single subroute per line, and there is no cost associated with allocating passengers to these lines. To tackle the RLPP, the authors propose a randomized rounding algorithm, building upon the scheme initially introduced by Fleischer et al. \cite{fleischer2011tight}. This algorithm yields a \(1-\frac{1}{e}-\epsilon\)-approximate solution, with the probability of exceeding the budget constraint capped at \(e^{-\frac{k}{3}\epsilon^2}\). Here, \( k = \frac{B}{\max_{l \in [L]} c_l} \), where \( B \) represents the budget, \( c_l \) is the cost associated with line \( l \), and \( L \) is the set of all lines. For a comprehensive discussion and additional related works in the domain of line planning, please see Section \ref{sec:applications}.

% It is worth noting, however, that the probability bound $e^{-\frac{k}{3}\epsilon^2}$ can exceed $\frac{1}{2}$ even when $k$ is as large as $\frac{1}{\epsilon^2}$. This trade-off is carefully examined in Remark \ref{remark:tradeoff}.

% It should be emphasized that despite the authors' claims to the contrary, the algorithm analysis presented for RLPP cannot be directly extended to GRLPP. This is due to a key difference between the two problems: in GRLPP, a passenger can generate different rewards when assigned to different lines, whereas in RLPP, this is not the case. We will delve into this distinction more deeply in Section \ref{sec:rlpp}.

% It is worth noting, however, that the probability bound $e^{-\frac{k}{3}\epsilon^2}$ can exceed $\frac{1}{2}$ even when $k$ is as large as $\frac{1}{\epsilon^2}$. This trade-off is carefully examined in \ref{sec:rlpp}. To balance the violation probability bound and the approximation ratio, one can set $\epsilon$ to $\frac{1}{k^\gamma}$, where $\gamma \leq \frac{1}{3}$. However, when $k=100$ and $\gamma = \frac{1}{3}$, the violation probability bound exceeds $0.2$. Moreover, the rounding process can generate infeasible solutions with higher objective values than feasible ones. When $k=100$ and $\gamma=\frac{1}{4}$, the approximation ratio decreases to less than $0.45$.

GAP in an online setting has also been studied by Alaei et al. \cite{alaei2013online}. They propose a $\left(1-\frac{1}{\sqrt{k}}\right)$-competitive algorithm under the assumption that no item takes up more than $\frac{1}{k}$ fraction of the capacity of any bin. In the setting of \cite{alaei2013online}, items arrive in an online manner; upon arrival, an item can be assigned to a bin or discarded; the
objective is to maximize the total reward of the assignment.  The size of an
item is revealed only after it has been assigned to a bin; the distribution of each item's reward and size is
known in advance.  The online algorithm is developed based on a generalization of the magician’s problem in \cite{alaei2014bayesian}, where it is used to {assign a set of items with limited supply to a set of buyers in order to maximize
the expected value of revenue or welfare}. 

%To enhance the readability of the paper, we will defer the discussion of the team orienteering problem and related work to Section \ref{sec:BMTOP}.

\subsection{Challenges}\label{sec:contribution}

 Our objective is to devise an algorithm that maintains a constant approximation ratio while being practically efficient. Balancing theoretical rigor with practical efficiency poses a significant challenge. Moreover, conventional approaches fail to yield a constant approximation ratio for our specific problem while meeting all practical constraints.

Firstly, while our problem shares similarities with the capacitated facility location problem (CFLP), for which \cite{wolsey1982maximising} provides a $1-\frac{1}{e}$ approximation algorithm based on the CFLP's submodular structure, our problem lacks this submodular property \cite[Proposition 4.10]{perivier2021real}. Consequently, methods aimed at maximizing submodular functions, as seen in works like \cite{lee2009non,kulik2011submodular,chekuri2012improved}, are not applicable to our context. This necessitates the development of innovative approaches tailored to the unique challenges discussed in this paper.

Secondly, the classical LP-based randomized rounding algorithms, proposed by \cite{fleischer2011tight,perivier2021real}, while simple, can have a non-trivial probability to produce solutions that violate the budget constraint. More specifically, the approach in \cite{perivier2021real}, an adaptation of the method from \cite{fleischer2011tight}, attempts to mitigate this by scaling down the LP solutions to manage the probability of budget constraint violations. While this method achieves a \(1-\frac{1}{e}-\epsilon\)-approximate solution, there remains a non-negligible probability, bounded by \(e^{-\frac{k}{3}\epsilon^2}\), that the budget constraints may still be breached. This probability fails to guarantee a constant approximation ratio of feasible solutions and can exceed \(\frac{1}{2}\) even when \(k\) is as large as \(\frac{1}{\epsilon^2}\).

To control this probability of violation, one might consider setting \(\epsilon\) to a relatively high value. However, this adjustment would negatively impact the solution quality, a dilemma extensively discussed in Remark \ref{remark:tradeoff}, which provides a detailed analysis and additional insights into this trade-off.

In summary, classical randomized rounding struggles to manage budget constraints and ensure that each item is assigned to at most one bin while still achieving a constant approximation ratio, especially when the algorithm must be computationally efficient to handle large-scale problems.

\subsection{Contributions}

We introduce a randomized approximation algorithm that is not only practically efficient but also maintains a constant approximation ratio. Our numerical experiments, conducted using real-world data, also outperform the previous work.

Our preliminary major findings were initially presented as an extended abstract at COCOON~\cite{jiang2022approximation}. It featured randomized approximation algorithms but lacked formal proofs and did not offer a constant approximation ratio. In this extended work, we have refined our algorithm to achieve a constant approximation ratio.

In Section \ref{sec:alg-1d}, we introduce our randomized approximation algorithm for solving GBAP, which achieves a constant approximation ratio. Specifically, our algorithm achieves an approximation ratio of $\frac{1}{8}$ for $k < 3$, and $\frac{k-1}{2k}\cdot \left(1-\frac{1}{\sqrt{k}}\right) > \frac{1}{8}$ for $k \geq 3$, where $k$ is defined in \eqref{eq:k-def}.  To the best of our knowledge, this is the first algorithm to achieve a constant approximation ratio for GBAP. 

Our analysis leverages two online optimization mechanisms detailed in Section \ref{sec:prelim}. The study of online knapsack problems in Section \ref{sec:online-knapsack} may be of independent interest. Furthermore, to overcome the limitations of classical randomized rounding, which struggles to simultaneously satisfy multiple constraints, we introduce a novel approach. Our method utilizes the realizations from randomized rounding to simulate two distinct online mechanisms. This innovative technique enables us to produce solutions that concurrently satisfy different constraints while maintaining theoretical guarantees. Furthermore, this method offers new insights into algorithm design for related complex optimization problems. 

% In Section \ref{sec:1d-no-magician}, we introduce a simpler algorithm for GBMAP with no assignment costs. This algorithm produces a $\beta\left(1-\frac{1}{\sqrt{\rho+3}}\right)(1-\frac{1}{\sqrt{k}})\frac{k-1}{k}-\epsilon$ approximation in expectation.

% We distinguish $\rho$-SAP and GBMAP by two crucial distinctions: (i) all bins are assumed to be available in $\rho$-SAP, and there is no cost associated with utilizing bins or assigning items to them, and (ii) each item can yield different rewards when assigned to the same bin. Thus, even when the budget is unlimited in GBMAP, i.e., $B=\infty$, the results for $\rho$-SAP cannot be directly extended to GBMAP in this case.

In Section \ref{sec:applications}, we apply our GBAP approximation algorithms to the RLPP . Using real-world data from New York City, we conduct numerical experiments that demonstrate improvements over previous approaches. These results highlight the practical significance of our algorithms in addressing complex transportation challenges.

\section{Overview of Algorithm Design and Analysis}\label{sec:overview} 

To provide a clear road-map of the technical discussion to follow, we first present an overview of our algorithm design and analysis, before delving into the specifics. 

In our algorithm design, we will consider two scenarios with respect to the ratio of the budget to the largest line cost ($k$ as defined in \eqref{eq:k-def}): \(k<3\) and \(k\geq 3\). For each scenario, we develop an alternative randomized procedure (see Lemma \ref{lem:kgeq3} and \ref{lem:kleq3}) that, while complex to implement, is theoretically analyzable. We derive an approximation ratio for each procedure and subsequently show that a simpler algorithm we design (see Algorithm \ref{alg-no-magician-1d}) does at least as well (up to a $\frac{1}{2}$ factor), thereby proving an approximation ratio for our simpler algorithm. Our proposed algorithm can be viewed as a simplified method that is derived from these two procedures. A general high level illustration of the techniques is given in the flow chart shown in in Figure \ref{fig:flow}.

\begin{figure}
    \centering
    \includegraphics[width=\linewidth]{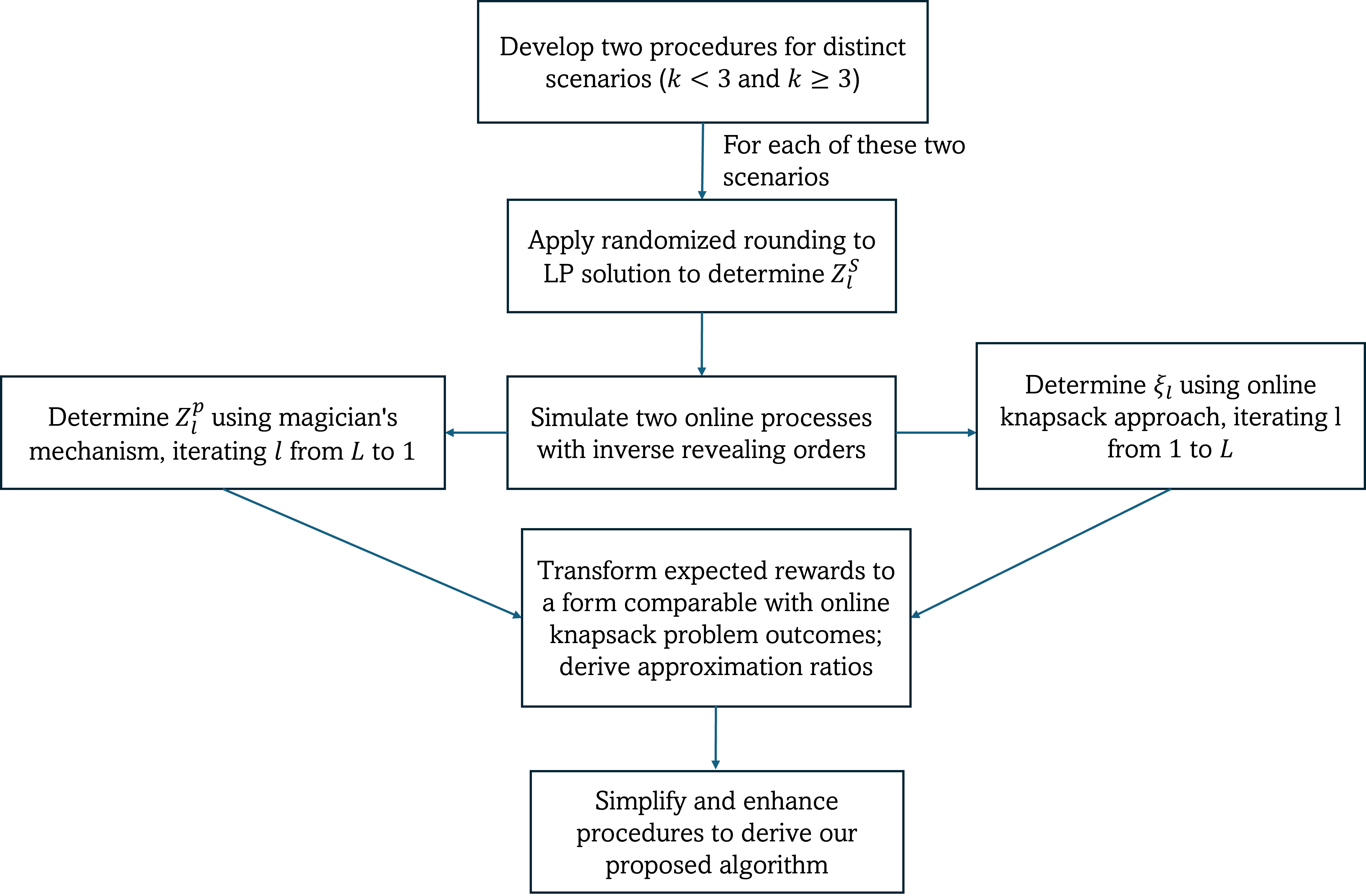}
    \caption{Overview of Algorithm Design and Analysis}
    \label{fig:flow}
\end{figure}

The design of these procedures follows a similar approach. Both procedures begin with randomized rounding based on the LP relaxation of the problem to select at most one \(S\in I_l\) for each \(l\in [L]\). This selection is represented by the indicator variables \(\left\{Z_l^S\right\}_{S\in I_l \cup \{\emptyset\}}\).  Specifically, let \(\widetilde{X}\) represent the LP solution. For each \(l \in [L]\), a set \(T_l\) is randomly selected from \(I_l \cup \{\emptyset\}\) according to the following distribution:
\begin{equation}
    P(T_l = S) = 
    \begin{cases}
        1 - \sum_{S \in I_l} \widetilde{X}_{lS} & \text{if } S = \emptyset, \\
        \widetilde{X}_{lS} & \text{if } S \in I_l,
    \end{cases}
\end{equation}
where the decision variable \(Z_l^S = 1\) if \(S = T_l\) and \(Z_l^S = 0\) otherwise, for each \(S \in I_l \cup \{\emptyset\}\). 

Subsequently, each item \(p\) can only be assigned to bins where the corresponding selected sets contain \(p\), i.e., \(p \in T_l\). Since \(T_l \in I_l \cup \{\emptyset\}\), where \(I_l\) represents the family of all nonempty feasible assignments of items to bin \(l\) as defined in \eqref{eq:Il}, this condition ensures that the capacity constraints are satisfied.
 We then focus on the other two types of constraints: (a) The budget constraint; (b) The restriction that each item can be assigned at most once. We will make two types of decisions: (a) Whether to utilize bin \(l\), represented by the indicator variable \(\xi_l\); (b) Whether to assign item $p$ to bin $l$ (not necessarily utilized), represented by the indicator variable $Z_l^p$. This assignment occurs only when three conditions are met simultaneously:
\begin{enumerate}
\item Bin $l$ is assigned a set $S$ containing $p$ ($Z_l^S=1$ and $p\in S$)
\item Bin $l$ is utilized ($\xi_l=1$)
\item Item $p$ is assigned to bin $l$ (not necessarily utilized) ($Z_l^p=1$)
\end{enumerate}
In other words, item $p$ is assigned to a utilized bin $l$ only when $Z_l^p Z_l^S \cdot \xi_l = 1$ for some $S$ such that $p\in S$. Importantly, in our algorithm design, $Z_l^p$ is set to be independent of both $Z_l^S$ and $\xi_l$. The rationale and implications of this independence will be discussed in detail later in this section.

The expectation of the rewards can then be expressed as:
\[
\E\left[\sum_{l\in [L]}\sum_{S\in I_l}\sum_{p\in S} v_{lp} Z_l^p Z_l^S \cdot \xi_l\right] = \sum_{l\in [L]}\sum_{S\in I_l}\sum_{p\in S} v_{lp} \E\left[Z_l^p Z_l^S \cdot \xi_l\right]
\]
This equality is due to the linearity of expectation and the linear objective function. It allows us to analyze each term $\E\left[Z_l^p Z_l^S \cdot \xi_l\right]$ respectively. 

The realizations of $Z_{l}^S$ have been determined through the previously described randomized rounding process. Subsequently, we determine $\xi_l$ and $Z_{l}^p$ by simulating two separate online decision-making processes based on these realizations.

To determine $\xi_l$ for $l\in [L]$, we simulate an online decision-making process using the realizations of $H_l := \sum_{S\in I_l}Z_l^S \in \{0,1\}$, revealing results as $l$ goes \emph{from $1$ to $L$}. This process is adapted from an online knapsack problem decision-making process (see Section \ref{sec:online-knapsack}). The process consists of $L$ stages. At the $l$th stage, the algorithm decides whether to utilize bin $l$ (set $\xi_l = 1$) based on the realizations of $H_{l'}$ for $l' \leq l$ and the remaining budget. For $k\geq 3$, it ensures that the overall costs of bin utilization are within the budget, i.e. $\sum_{l\in [L]}c_l\cdot \xi_l\leq B$, while for $k<3$, the costs are approximately within the budget. It is important to note that, in contrast to these two alternative procedures, our proposed algorithm ensures strict adherence to the budget constraint in both cases.

For each item \( p \), to determine \( Z_l^p \) for \( l \in [L] \), we simulate an online decision-making process using the realizations of \( \widetilde{Y}_{lp} := \sum_{\substack{S \in I_l: p \in S}} Z_l^S\in\{0,1\} \). This process iterates over \( l \) in \emph{reverse order, from \( L \) to \( 1 \)}, and is based on the magician's mechanism (see Section \ref{sec:magician}). The process consists of \( L \) stages. At the \( (L-l+1) \)th stage\footnote{We use \( (L-l+1) \) instead of \( l \) to align with the descending order of \( l \), which simplifies subsequent index references of this sentence.},
the algorithm decides whether to assign item \( p \) to bin \( l \) (i.e., set \( Z_l^p = 1 \)) based on the realizations of \( \widetilde{Y}_{l'p} \) and the previous decisions \( Z_{l'}^p \) for \( l' > l \). 
 It is important to observe that this iteration order for determining $Z_l^p$ is inverse to the one used for determining $\xi_l$. This inverse ordering plays a crucial role in ensuring the independence of these decision variables in our analysis.

The magician's mechanism ensures that (i) $\sum_{l\in[L]}\sum_{S:p\in S}Z_{l}^pZ_l^S\leq 1$, (ii) $Z_l^p$ is independent of $Z_l^S$, and (iii) $\E[Z_l^p]\geq\frac{1}{2}$. Note that point (i) ensures that item $p$ will be assigned to at most one bin. Due to the inverse revealing order for determining $\xi_l$, $Z_l^p$ is also independent of $\xi_l$. Consequently:

\[ \E\left[Z_l^p Z_l^S \cdot \xi_l\right] =  \E\left[Z_l^p\right]\cdot \E\left[ Z_l^S \cdot \xi_l\right]\geq \frac{1}{2}\E\left[ Z_l^S \cdot \xi_l\right]\].

Therefore, 

\begin{align}\label{eq:reduction}
    \E\left[\sum_{l\in [L]}\sum_{S\in I_l}\sum_{p\in S} v_{lp} Z_l^p Z_l^S \cdot \xi_l\right] \geq \sum_{l\in [L]}\sum_{S\in I_l}\sum_{p\in S} \frac{v_{lp}}{2} \E\left[ Z_l^S \cdot \xi_l\right] =\frac{1}{2}\E\left[\sum_{l\in [L]}\sum_{S\in I_l}\sum_{p\in S} v_{lp} Z_l^S \cdot \xi_l\right]
\end{align}

\noindent The last term is then comparable to the outcome of the online knapsack problem decision-making process defined in Section \ref{sec:online-knapsack}, which helps us derive the desired result for these two alternative procedures. Finally, we compare our algorithm (Algorithm \ref{alg-no-magician-1d}) to these two procedures to establish the desired approximation ratio.

Having provided an overview of our approach, we now outline the structure of the two subsequent sections. In Section \ref{sec:prelim}, we introduce two key online mechanisms: the magician's mechanism and online fractional knapsack problems. These mechanisms are essential for developing the two alternative procedures, and we present their relevant results. Section \ref{sec:alg-1d} then details our proposed algorithm, along with the two alternative procedures and their associated theoretical outcomes.

\section{Essential Mechanisms for Algorithm Design and Analysis}\label{sec:prelim}

As discussed in Section \ref{sec:overview}, our algorithm's analysis is based on two key online optimization mechanisms. The first is the \emph{generalized \( \gamma \)-conservative magician}, originally proposed by Alaei et al.\ \cite{alaei2014bayesian}. This mechanism serves a dual purpose: it determines $Z_l^p$, as outlined in Section \ref{sec:overview}, and provides a part of the theoretical foundation for proving results in the second mechanism. We present this mechanism here for completeness. The second mechanism, which we will prove the efficacy of subsequently, adapts concepts from online knapsack problems and may be of independent interest to the broader optimization community.

\subsection{Generalized Magician's Problem \cite[Section 7]{alaei2014bayesian}}\label{sec:magician}

\begin{definition}(the generalized magician's problem). A magician is presented with a sequence of $L$ boxes one by one in an online fashion. The magician has $\delta$ units of mana. The magician can only open box $i$ if she has at least $1$ unit of mana. If box $i$ is opened, the magician loses a random amount of mana $X_i\in [0,1]$ drawn from a distribution specified on box $i$ by its cumulative distribution function $F_{X_i}$.  It is assumed that $\E\left[\sum_{i=1}^LX_i\right]\leq \delta$. The magician wants to open each box with ex ante probability at least $\gamma$, i.e. not conditioned on $X_i$ for $i=1,\ldots, L$, for a constant $\gamma\in [0,1]$ as large as possible.
\end{definition}

\noindent The problem can be solved with a near-optimal solution using the following mechanism. 

\begin{definition}\label{def:magician}(generalized $\gamma$-conservative magician) The magician %computes a sequence of thresholds $\theta_2,\ldots, \theta_L\in \N_0$ in a dynamic programming fashion, and 
makes a decision about each box as follows: Let {the random variable} $W_i$ denote the amount of mana lost prior to seeing the $i$th box, and $F_{W_i}(w):=P[W_i\leq w]$ denote the ex ante CDF of random variable $W_i$.  Define random binary variable $S_i$ conditional on $W_i$ as follows:
\begin{align*}
    P[S_i=1|W_i]&=\begin{cases}
          1,&W_i<\theta_i\\
          (\gamma-F^-_{W_i}(\theta_i))/(F_{W_i}(\theta_i)-F^-_{W_i}(\theta_i)), &W_i=\theta_i\\
          0, &W_i>\theta_i.
    \end{cases}\\
    \theta_i&=\min\{w|F_{W_i}(w)\geq \gamma\}\mbox{ and }F^-_{W_i}(w)=P[W_i<w].
\end{align*}
Then given $W_i$ and before seeing the $i$th box, the magician will decide to open the $i$th box if $S_i=1$.
\end{definition}

Note that $W_1=0$ with probability $1$, and therefore $\theta_1=0$. Then the CDF of $W_{i+1}$ and $\theta_{i+1}$ can be computed based on $W_{i}$ and $\theta_i$ in with dynamic programming. For more details, please see \cite[Section 7]{alaei2014bayesian}. 

One observation is that $S_i$ is independent to $X_i$ for all $i$. 
\begin{theorem}\label{thm:magician}
\cite[Theorem 7.3]{alaei2014bayesian} For any $\gamma\leq 1-\frac{1}{\sqrt{\delta}}$, we have $\theta_i\leq \delta-1$ for all $i$, and a generalized $\gamma$-conservative magician with $k$ units %\cite[Footnote 10]{alaei2013online}.} 
 of mana opens each box with an ex ante probability of exactly $\gamma$, i.e., $P(S_i=1)=\gamma$. If all $X_i$ are Bernoulli random variables, i.e., $X_i\in \{0,1\}$ for all $i$, and $\delta$ is an integer, then for any $\gamma\leq 1-\frac{1}{\sqrt{\delta+3}}$, we have $\theta_i\leq \delta-1$ for all $i$, and a generalized $\gamma$-conservative magician with $\delta$ units of mana opens each box with an ex ante probability of exactly $\gamma$. 
\end{theorem}

\begin{remark}
In Theorem \ref{thm:magician}, $\delta$ can be non-integer when $X_i$ is not restricted to be Bernoulli variable as mentioned in \cite{alaei2013online} and by its proof. 
\end{remark}

\noindent {\bf Application in determining $Z_l^p$:} For each item $p$, we construct a generalized $\frac{1}{2}$-conservative magician with one unit of mana. This magician faces $L$ boxes, where the mana cost for opening each box follows a Bernoulli distribution. This distribution is determined by the distribution of $Z_l^S$ resulting from the randomized rounding process. The magician's decisions on whether to open these boxes directly correspond to the determination of $Z_l^p$. For a detailed description of this process, please refer to Step 3 of the procedures outlined in Lemma \ref{lem:kgeq3} and Lemma \ref{lem:kleq3}.

\subsection{Online Fractional Knapsack Problems}\label{sec:online-knapsack}

\begin{definition}[Online Fractional Knapsack Problem with Pre-arranged Arrival Order]\label{def:online-knapsack}
Consider a sequence of \( n \) independent rounds. In each round \( i \), an object is characterized by a positive reward \( r_i \), a non-negative weight \( w_i \), and a positive appearance probability \( p_i \). Additionally, there is a specified capacity \( \delta \geq 1 \), ensuring that the cumulative weight of the objects accepted does not exceed this capacity. The appearance of each object follows a Bernoulli distribution. The problem is subject to two constraints:
\[
\sum_{i=1}^{n} w_i \cdot p_i \leq \delta, \quad 0 < w_i \leq 1.
\]

In this setting, we assume full knowledge of the object distribution. At the \( i \)-th round, we observe the appearance of object \( i \) before deciding a fraction \( \beta_i \in [0, 1] \) of object \( i \) to accept. Acceptance is conditioned on the remaining capacity being at least \( \beta_i \cdot w_i \). Accepting the object results in a reduction of the remaining capacity by \( \beta_i \cdot w_i \) and a secured reward of \( \beta_i \cdot r_i \). All decisions, once made, are irreversible.

It is assumed that the sequence of object appearances can be pre-determined before the initiation of the rounds.

The goal is to devise an algorithm that optimizes an efficiency coefficient \( \alpha \in (0, 1) \) such that the expected reward, at the end of \( n \) rounds, is \( \alpha \cdot \sum_{i=1}^{n} r_i \cdot p_i \).
\end{definition}

\begin{remark}
    The term ``object" is used here instead of ``item" to provide a distinction from the ``item" used in GBAP. Specifically, in the analysis presented in Section \ref{sec:alg-1d}, the allocation of objects in this online knapsack problem will be compared to the utilization of bins in GBAP.
\end{remark}

\begin{lemma}[Expected Reward for Prearranged Arrival Order]\label{expected-reward-knapsack}
Consider the problem defined in Definition \ref{def:online-knapsack}. 
Assume the order of object appearances is prearranged such that

\begin{align}\label{eq:pre-order}
    \frac{r_1}{w_1} \geq \frac{r_2}{w_2} \geq \ldots \geq \frac{r_n}{w_n}.
\end{align}

Furthermore, let the strategy be to accept objects fully in order of appearance until the remaining capacity is insufficient to accept the next appearing object fully. At this point, accept this next appearing object fractionally to exactly exhaust the remaining capacity. Under these conditions, the expected reward generated at the end of \( n \) rounds is
\[
\max\left\{\frac{1}{2},\left(1-\frac{1}{\sqrt{\delta}}\right)\right\} \cdot \sum_{i \in [n]} r_i \cdot p_i.
\]
\end{lemma}

\begin{proof}
We begin by establishing the  ratio of \(\left(1 - \frac{1}{\sqrt{\delta}}\right)\). According to Definition \ref{def:magician} and Theorem \ref{thm:magician}, the problem defined in Definition \ref{def:online-knapsack} admits a solution with this ratio when using the Magician's mechanism. Note that in this case, the acceptance fraction $\beta_i\in \{0,1\}$ for all $i\in [n]$ due to the definition of Magician's mechanism.  It is evident that the strategy outlined in Lemma \ref{expected-reward-knapsack} outperforms the Magician's mechanism in terms of efficiency and reward due to the pre-arranged order of objects as in \eqref{eq:pre-order}. Therefore, this concludes the proof for this part.

As for the ratio of $\frac{1}{2}$, we will conduct the similar comparison, but with the help of the method in Manshadi et al. \cite{manshadi2021fair}. In  their work, they defined a method such that at the end of $n$ rounds, it achieves that $\E\left[\min_{i\in [n]}\beta_i\right]\geq \frac{1}{2}$. For more details, please see Theorem 2 in \cite{manshadi2021fair} when $\mu \leq 1$. This means that the method finally yields expected reward at least $$\frac{1}{2}\cdot \sum_{i \in [n]} r_i \cdot p_i.$$
Again, the greedy strategy outlined in Lemma \ref{expected-reward-knapsack} outperforms this method in terms of  reward due to the pre-arranged order of objects as in \eqref{eq:pre-order}. Therefore, this concludes the proof.
\end{proof}

\noindent {\bf Application in determining $\xi_l$ and deriving approximation ratio:} To apply the results of Lemma \ref{expected-reward-knapsack} in determining whether to utilize a bin $l$ (i.e., determine $\xi_l$ as mentioned in Section \ref{sec:overview}), we simulate the online knapsack problem in Algorithm \ref{alg-no-magician-1d}. Specifically, after a randomized rounding process to realize $Z_l^S$, we reindex the $L$ bins following the rule in \eqref{eq:pre-order} of Algorithm \ref{alg-no-magician-1d}. We then simulate the online knapsack problem as described in Section \ref{sec:overview}. This approach serves two purposes: (1) it enables us to determine $\xi_l$ by simulating the decision-making process for the online knapsack problem, and (2) it allows us to compare the last term in \eqref{eq:reduction} to the outcome of the online knapsack problem, enabling us to leverage the bound established in Lemma \ref{expected-reward-knapsack} to derive the desired approximation ratio.

\section{Approximation Algorithm} \label{sec:alg-1d}
In this section, we consider solving GBAP. Let $k$ be the ratio of the budget to the maximum bin cost: $$k:=\frac{B}{\max_{l\in[L]}c_{l}}.$$

\noindent For the sake of simplicity and without loss of generality, we can scale the budget and costs to adhere to the following assumption:
\begin{assumption}\label{assumption:budget-scaling}
The given budget is $k$, and for all $l\in [L]$, $c_{l}\leq 1$, with $\max_{l\in[L]}c_{l}=1$.
\end{assumption}

% \begin{subequations}\label{eq:dual}
% \begin{align}
%     \underset{\{q_l,\lambda_p,\alpha_k\}}{\min}~~~~~~~&\sum_{l\in[L]}q_l+\sum_{p\in[P]}\rho\lambda_p+k\alpha\\
%     s.t~~~~~~~~~~~&q_l+\alpha c_{l}\geq\sum_{p\in S}(v_{lp}-\lambda_p-\alpha c_{lp})~~~~~~\forall l\in [L],S\in I_l\\
%     &q_l,\lambda_p,\alpha\geq 0~~~\forall l\in[L],~p\in[P]
% \end{align}
% \end{subequations}

% For the rest of the paper, we let $[L]=\{1,\ldots, L\}$.

\noindent The LP relaxation  can be solved in polynomial time by solving its dual with ellipsoid method as in \cite{fleischer2011tight,perivier2021real}.

\begin{alg}\label{alg-no-magician-1d}
~
\begin{enumerate}
    %\item Solve (\ref{eq:LP}) and let $\widehat{X}$ be the solution to it. 
    %\item For each $l$, let $I^*_l=\{S\in I_l:\widehat{X}_{lS}>0\}$.
    %\item Assume $\frac{\sum_{S\in I_i}\sum_{p\in S}\widehat{X}_{iS}v_{ip}}{c'_i\sum_{S\in I_i}\widehat{X}_{iS}}\geq \frac{\sum_{S\in I_j}\sum_{p\in S}\widehat{X}_{jS}v_{jp}}{c'_j\sum_{S\in I_j}\widehat{X}_{jS}}$ for $1\leq i<j\leq L$.
    \item Find a solution to the LP relaxation of (\ref{eq:2d-IP-config-original}) and let $\widetilde{X}$ be the solution. 
    \item Reorder the bins so that 
    \begin{align}\label{eq:reorder}
    \frac{\sum_{S\in I_l}\left(\widetilde{X}_{lS}\sum_{p\in S}v_{lp}\right)}{c_{l}\sum_{S\in I_l}\widetilde{X}_{lS}}
    \geq
    \frac{\sum_{S\in I_{l'}}\left(\widetilde{X}_{l'S}\sum_{p\in S}v_{l'p}^{S}\right)}{c_{l'}\sum_{S\in I_{l'}}\widetilde{X}_{l'S}}~~~~~~\mbox{ for } 1\leq l<l'\leq L.
\end{align}
Note that here if the denominators are zeros, the corresponding numerators are also zeros, and the corresponding bins can be ignored in the following rounding process. Thus,  we can define $\frac{0}{0}:=0$ in (\ref{eq:reorder}).
    \item Independently for each $l\in[L]$, %flip a two-sided coin with $\sum_{S\in I_l}\widehat{X}_{lS}$ as the probability to be the head. If the result is tail, we select an empty set $\emptyset$. Otherwise, randomly select $S\in I_l$ with probability $\frac{\widehat{X}_{lS}}{\sum_{S'\in I_l}\widehat{X}_{lS'}}$. Let the selected set be denoted by $S_l$. The distribution of $S_l$ can be written as 
    randomly select a set $T_l$ from $I_l\cup\{\emptyset\}$ following the distribution below:
   \begin{equation*}
        P({T}_l=S)=
        \begin{cases}
      1-\sum_{S\in I_l}\widetilde{X}_{lS} & \text{if }S=\emptyset\\
      \widetilde{X}_{lS} & \text{if }S\in I_l
    \end{cases} .
    \end{equation*}

Let the (random) decision  variable be $Z_l^S$, i.e., $Z_l^S=1$ if $S=T_l$, and $Z_l^S=0$ otherwise for $S\in I_l\cup\{\emptyset\}$.
    \item Temporarily assign all the elements in $T_l$ to bin $l$ for each $l\in [L]$.
    \item For each item $p$, retain it in the bin where it yields the highest reward $v_{lp}$ and remove it from others.
    \item Open all the bins that have been assigned with at least one item, and close all the other bins. If this assignment does not exceed the budget, then we take this assignment as the final solution.
    \item Otherwise, discard this assignment. We construct a new assignment.
    \item Let $B_0=k$. Let $l$ be from $1$ to $L$. At $l$th iteration, let $B_l:=B_{l-1}-{c}_{l}\sum_{S\in I_l}Z_l^S\geq 0$.
    \item If $k\geq 3$, let $\mathcal{L}^*:=\{l:B_l\geq 0,~\sum_{S\in I_l}Z_l^S>0\}$. Assign $T_l$ to $l$ for each $l\in \mathcal{L}^*$ and open it. For each item $p$, retain it in the bin where it yields the highest reward $v_{lp}$ and remove it from others.
    \item If $k<3$, let $\mathcal{L}^*:=\{l:B_{l-1}> 0,~\sum_{S\in I_l}Z_l^S>0\}$. Assign $T_l$ to $l$ for each $l\in \mathcal{L}^*$ and \emph{temporarily} open it. For each item $p$, retain it in the bin where it yields the highest reward $v_{lp}$ and remove it from others. Let $l^* = \max_{l\in \mathcal{L}^*}l$. Then we open the bins in either $\mathcal{L}^*\backslash\{l^*\}$ or in $\{l^*\}$, whichever generated more rewards.
\end{enumerate}
\end{alg}

Our proof is divided into two cases: $k \geq 3$ and $k < 3$. For both scenarios, we employ a two-step approach. First, we analyze the expected reward generated by a specific procedure of assignment based on two mechanisms: the magician's mechanism defined in Definition \ref{def:magician} and the mechanism adapted from the online fractional knapsack problem discussed in Section \ref{sec:online-knapsack}. Second, we compare the rewards generated by Algorithm \ref{alg-no-magician-1d} with those from this constructed procedure.

We begin by addressing the case where $k \geq 3$. For this scenario, we construct a specific procedure and provide its analysis in the following lemma.

\begin{lemma}\label{lem:kgeq3}
    Assume $k\geq 3$. With the same notations as in Algorithm ~\ref{alg-no-magician-1d}, we define a procedure as follows:
\begin{enumerate}
    \item For each $l \in [L]$, let $H_l = \sum_{S \in I_l} Z_l^S$, which indicates whether any set $S$ has been selected for bin $l$ in the randomized rounding process. Define $\xi_l = 1$ if $B_{l} \geq 0$ and $H_l = 1$. Otherwise, set $\xi_l = 0$. Here $B_{l}$ represents the remaining budget after considering the first $l$ bins. Thus, $\xi_l = 1$ indicates that bin $l$ is utilized, which occurs when there is sufficient budget ($B_{l} \geq 0$) and a set has been selected for this bin ($H_l = 1$).
    
    \item For every $l \in [L]$ and $p \in [P]$, introduce a variable $\widetilde{Y}_{lp}$, given by
    \[
    \widetilde{Y}_{lp} = \sum_{\substack{S \in I_l: p \in S}} Z_l^S.
    \]
    It indicates whether item $p$ is included in the selected set for bin $l$, and its probability distribution is
    \[
    P(\widetilde{Y}_{lp} = y) = \begin{cases}
        1 - \sum_{\substack{S \in I_l:  p \in S}} \widetilde{X}_{lS}, & \text{if } y = 0, \\
        \sum_{\substack{S \in I_l: p \in S}} \widetilde{X}_{lS}, & \text{if } y = 1.
    \end{cases}
    \]
    
        \item For each item $p\in [P]$,  create a generalized $\frac{1}{2}$-conservative magician with $1$ unit of  mana (called \emph{type-p} magician). The magician of item $p$ is presented with a sequence of $L$ \emph{type-$p$} boxes \emph{from box $L$ to box $1$}. The distribution of $\widetilde{Y}_{lp}$ is written on type-$p$ box $l$. The amount of mana the magician will lose if open type-$p$ box $l$ is equal to the value of $\widetilde{Y}_{lp}\in \{0,1\}$. The magician will decide whether to open type-$p$ box $l$ following the mechanism defined in Definition \ref{def:magician}.  Let the (random) decision  variable be $Z_{l}^p$, i.e., $Z_l^p=1$ if the magician for $p$ decides to open type-$p$ box $l$, and  $Z_l^p=0$ otherwise.
\end{enumerate}

\noindent If we utilize bin $l$ when $\xi_l = 1$, and assign item $p$ to bin $l$ whenever $Z_l^p\cdot  \widetilde{Y}_{lp} \cdot \xi_l = 1$, then the assignment is feasible, which means that:
\begin{enumerate}[(a)]
\item Each passenger is assigned to at most one utilized bin.
\item The total cost of all utilized bins is within the given budget $k$.
\item The capacity constraints are respected.

\end{enumerate}
Moreover, the expected approximation ratio derived from this assignment is at least $\left(\frac{k-1}{2k}\right)\cdot\left(1-\frac{1}{\sqrt{k}}\right)>\frac{1}{8}$.
\end{lemma}

\begin{proof}
We first prove its feasibility. We start with point (a). 
%Note that if we assign item \( p \) to bin \( l \) when \( Z_l^p \cdot \widetilde{Y}_{lp} = 1 \) (rather than when \( Z_l^p \cdot \widetilde{Y}_{lp}\cdot \xi_l = 1 \)), then this assignment of item \( p \) to bin \( l \) is equivalent to spending \( 1 \) unit of mana when opening the type-\( p \) box \( l \).
In Step 3, $Z_l^p$ represents the decision of the type-$p$ magician on whether to open the type-$p$ box with index $l$. Meanwhile, $\widetilde{Y}_{lp}$ indicates whether opening the box requires spending 1 unit of mana. Based on the definition of the magician's mechanism in Definition \ref{def:magician} and the results stated in Theorem \ref{thm:magician}, the type-\( p \) magician will spend at most one unit of mana during this process. Therefore, \(\sum_{l=1}^L Z_l^p \cdot \widetilde{Y}_{lp} \leq 1\). Since \( Z_l^p \cdot  \widetilde{Y}_{lp} \cdot \xi_l \leq Z_l^p \cdot  \widetilde{Y}_{lp} \), so if we assign item \( p \) to bin $l$ when \(Z_l^p \cdot  \widetilde{Y}_{lp} \cdot \xi_l =1\), it can can only be assigned at most once through this process. As for point (b) and (c), they can be checked directly by definition.

The expected reward derived from the procedure in Lemma \ref{lem:kgeq3} can be expressed as:
\begin{align}
    \E\left[\sum_{l\in[L]}\sum_{p\in [P]}v_{lp}Z_l^p\cdot \widetilde{Y}_{lp}\cdot \xi_l\right]
    &=\E\left[\sum_{l\in[L]}\sum_{S\in I_l}\sum_{p\in S}v_{lp}Z_l^pZ_l^S\cdot \xi_l\right]= \sum_{l\in[L]}\sum_{S\in I_l}\sum_{p\in S}v_{lp}\E\left[Z_l^pZ_l^S\cdot \xi_l\right]\nonumber\\
    &=\sum_{l\in[L]}\sum_{S\in I_l}\sum_{p\in S}v_{lp}\E\left[Z_l^p\right]\E\left[Z_l^S\cdot \xi_l\right]\label{eq:independence-2}\\
    &\geq \frac{1}{2}\sum_{l\in[L]}\sum_{S\in I_l}\sum_{p\in S}v_{lp}\E\left[Z_l^S\cdot \xi_l\right]\label{eq:Zlp-replace}\\
   &=\frac{1}{2}\E\left[\sum_{l\in[L]}\sum_{S\in I_l}\sum_{p\in S}v_{lp}Z_l^S\cdot \xi_l\right] = \frac{1}{2}\sum_{l\in[L]}\E\left[\sum_{S\in I_l}\sum_{p\in S}v_{lp}Z_l^S\Bigg\vert \xi_l = 1\right]\cdot P\left[\xi_l=1\right]\nonumber\\
   & = \frac{1}{2}\sum_{l\in[L]}\E\left[\sum_{S\in I_l}\sum_{p\in S}v_{lp}Z_l^S\Bigg\vert H_l = 1\right]\cdot P\left[\xi_l=1\right]\label{eq:xi=H}
\end{align}

\noindent Eq. (\ref{eq:independence-2}) results from observing that $Z_l^p$ is independent of $\xi_l$, and $Z_l^S$ for each $l\in[L]$ and $S\in I_l$. This independence arises because the magician's mechanism ensures that $Z_l^p$ only depends on $\widetilde{Y}_{l'}^p$ and $Z_{l'}^p$ for $l' \in \{l+1, l+2, \ldots, L\}$ as defined in Step 3 of the procedure, while $\xi_l$ depends on $Z_{l'}^S$ for $l' \in \{1, 2, \ldots, l\}$ and $S\in I_l$. \eqref{eq:Zlp-replace} occurs because $Z_l^p$ represents the decision of a type-$p$ magician to open the box or not, constrained by having only one unit of mana, leading to $E[Z_l^p]\geq \frac{1}{\sqrt{1+3}}=\frac{1}{2}$ according to Theorem \ref{thm:magician}.

As for Eq. (\ref{eq:xi=H}), it is due to $\xi_l$ depending only on $H_l$ and $H_{l'}$ for $l'<l$, and $Z_l^S$ being independent of $H_{l'}$ for $l'<l$.

Let us define 
\[ r_l := \E\left[\sum_{S \in I_l}\sum_{p \in S} v^S_{lp} Z_l^S \mid H_l = 1\right] = \frac{\sum_{S \in I_l}\left(\widetilde{X}_{lS} \sum_{p \in S} v_{lp}\right)}{\sum_{S \in I_l} \widetilde{X}_{lS}}. \] 
Subsequently, Eq.~\eqref{eq:xi=H} can be expressed as

\begin{align}\label{eq:reduce-online-knapsack}
    \frac{1}{2}\sum_{l \in [L]} \E\left[r_l \cdot \xi_l\right] = \frac{1}{2} \E\left[\sum_{l \in [L]} r_l \cdot \xi_l\right].
\end{align}

To establish the lower bound of Eq.~(\ref{eq:xi=H}), we construct an instance of the online fractional knapsack problem, as in Definition~\ref{def:online-knapsack}, consisting of \( L \) rounds. In each round \( l \), an object appears with a probability \( P(H_l = 1) \), possessing a weight \( c_l \) and a reward \( r_l \). The initial capacity is equal to the budget $k$. As specified in Step 2 of Algorithm~\ref{alg-no-magician-1d}, it satisfies that \( \frac{r_1}{c_1} \geq \frac{r_2}{c_2} \geq \ldots \geq \frac{r_L}{c_L} \), allowing us to apply Lemma~\ref{expected-reward-knapsack}.

In the constructed instance, if we fully accept each appearing object in sequence until the remaining capacity is inadequate for the full acceptance of the next appearing object, and then accept the next object fractionally to precisely exhaust the remaining capacity, then, according to Lemma~\ref{expected-reward-knapsack} with $k\geq 3$, the expected rewards obtained in this online knapsack problem are at least:

\begin{align*}
    \left(1-\frac{1}{\sqrt{k}}\right)\sum_{l\in [L]}r_l\cdot P(H_l=1)=\left(1-\frac{1}{\sqrt{k}}\right)\sum_{l\in[L]}\sum_{p\in [P]}\sum_{S\in I_l:p\in S}v_{lp}\widehat{X}_{lS}.
\end{align*}

If we instead only accept objects fully and stop when the remaining capacity is insufficient for the next appearing object (i.e., we do not accept any object fractionally), the reward equals $\sum_{l\in [L]}r_l\cdot \xi_l$. Since $c_l\leq 1$, the capacity is $k$, and $\frac{r_1}{c_1}\geq \frac{r_2}{c_2}\geq \ldots\geq \frac{r_L}{c_L}$, then we have that 

\begin{align}
    \E\left[\sum_{l\in [L]}r_l\cdot \xi_l\right] &\geq \frac{k-1}{k}\cdot\left(1-\frac{1}{\sqrt{k}}\right)\sum_{l\in [L]}r_l\cdot P(H_l=1)\nonumber\\
    &= \frac{k-1}{k}\cdot\left(1-\frac{1}{\sqrt{k}}\right)\sum_{l\in[L]}\sum_{p\in [P]}\sum_{S\in I_l:p\in S}v_{lp}\widehat{X}_{lS}.\label{eq:discard-last-fractional}
\end{align}

Summarizing \eqref{eq:xi=H},\eqref{eq:reduce-online-knapsack} and \eqref{eq:discard-last-fractional}, we conclude that

$$\E\left[\sum_{l\in[L]}\sum_{S\in I_l}\sum_{p\in S}v_{lp}Z_l^pZ_l^S\cdot \xi_l\right] =\frac{1}{2}\E\left[\sum_{l\in [L]}r_l\cdot \xi_l\right] \geq \frac{k-1}{2k} \left(1-\frac{1}{\sqrt{k}}\right)\sum_{l\in[L]}\sum_{p\in [P]}\sum_{S\in I_l:p\in S}v_{lp}\widehat{X}_{lS}.$$

\noindent This completes the proof.
\end{proof}

\begin{theorem}\label{thm:kgeq3}
    Assume $k\geq 3$. The solution given by Algorithm \ref{alg-no-magician-1d} is feasible. The expected approximation ratio of the solution given by Algorithm \ref{alg-no-magician-1d} is at least $\left(\frac{k-1}{2k}\right)\cdot\left(1-\frac{1}{\sqrt{k}}\right)>\frac{1}{8}$.
\end{theorem}

\begin{proof}
The feasibility can be verified directly. Given the same realization of $Z_l^S$ in Step 3 of Algorithm \ref{alg-no-magician-1d}, both Algorithm \ref{alg-no-magician-1d} and the procedure defined in Lemma \ref{lem:kgeq3} will utilize identical bins.

Moreover, under this realization, item $p$ can only be assigned to bins where $Z_{l}^S=1$ and $p\in S$. This condition holds true for both Algorithm \ref{alg-no-magician-1d} and the procedure in Lemma \ref{lem:kgeq3}.

Given the realizations of $Z_l^S$,  the optimal reward is achieved by retaining each item $p$ in the bin where it yields the highest reward $v_{lp}$, subject to $Z_{l}^S=1$ and $p\in S$. This is precisely what Step 9 of Algorithm \ref{alg-no-magician-1d} accomplishes.

Therefore, the expected reward of Algorithm \ref{alg-no-magician-1d} will be at least as high as that derived from the procedure defined in Lemma \ref{lem:kgeq3}.
\end{proof}

The above proof is constructed by creating a procedure that produces a feasible assignment solution, then comparing Algorithm \ref{alg-no-magician-1d} with it to derive the desired result. For the case when $k < 3$, we will instead produce a procedure that generates an assignment solution that may violate the budget constraint. We will then compare this with Algorithm \ref{alg-no-magician-1d} to derive the desired result. Although the approaches differ in terms of solution feasibility of the constructed procedures, the key ideas remain similar. Both rely on the magician's mechanism and the results from the online fractional knapsack problem presented in Section \ref{sec:prelim}.

\begin{lemma}\label{lem:kleq3}
    Assume $k< 3$. With the same notations as in Algorithm ~\ref{alg-no-magician-1d}, we define a procedure as follows:
\begin{enumerate}
    \item For each $l \in [L]$, let $H_l = \sum_{S \in I_l} Z_l^S$. Set $\xi_l = 1$ if $B_{l-1} \geq 0$ and $H_l = 1$. Otherwise, set $\xi_l = 0$. This ensures that bin $l$ is utilized ($\xi_l = 1$) only when the remaining budget for $l$ is non-negative ($B_{l-1} \geq 0$) and a set has been selected for this bin ($H_l = 1$).
    
    \item For every $l \in [L]$ and $p \in [P]$, introduce a random variable $\widetilde{Y}_{lp}$, given by
    \[
    \widetilde{Y}_{lp} = \sum_{\substack{S \in I_l: p \in S}} Z_l^S.
    \]
    Its probability distribution is
    \[
    P(\widetilde{Y}_{lp} = y) = \begin{cases}
        1 - \sum_{\substack{S \in I_l:  p \in S}} \widetilde{X}_{lS}, & \text{if } y = 0, \\
        \sum_{\substack{S \in I_l: p \in S}} \widetilde{X}_{lS}, & \text{if } y = 1.
    \end{cases}
    \]
    
        \item For each item $p\in [P]$,  create a generalized $\frac{1}{2}$-conservative magician with $1$ unit of  mana (called \emph{type-p} magician). The magician of item $p$ is presented with a sequence of $L$ \emph{type-$p$} boxes \emph{from box $L$ to box $1$}. The distribution of $\widetilde{Y}_{lp}$ is written on type-$p$ box $l$. The amount of mana the magician will lose if open type-$p$ box $l$ is equal to the value of $\widetilde{Y}_{lp}\in \{0,1\}$. The magician will decide whether to open type-$p$ box $l$ following the mechanism defined in Definition \ref{def:magician}.  Let the (random) decision  variable be $Z_{l}^p$, i.e., $Z_l^p=1$ if the magician for $p$ decides to open type-$p$ box $l$, and  $Z_l^p=0$ otherwise.
\end{enumerate}

\noindent If we utilize bin $l$ when $\xi_l = 1$, and assign item $p$ to bin $l$ whenever $Z_l^p\cdot  \widetilde{Y}_{lp} \cdot \xi_l = 1$ for some $S \in I_l$ and $p\in S$, then the assignment satisfies:
\begin{enumerate}[(a)]
\item Each passenger is assigned to at most one utilized bin.
\item Let the total cost of all utilized bins be $k'$ and define $l^* = \max\{l\in [L]: \xi_l=1\}$. Then we have $k'-c_{l^*} \leq k$. In other words, if we remove the utilized bin with the largest index, the total cost of the remaining bins is within the given budget $k$.
\item The capacity constraints are respected.

\end{enumerate}
Moreover, the expected approximation ratio derived from this assignment is at least $\frac{1}{4}$.
\end{lemma}

\begin{proof} As the proof employs reasoning similar to that used in Lemma \ref{lem:kgeq3}, we have placed the full proof in Appendix \ref{appendix:proof-of-lem-kleq} to maintain the flow of the main text.
\end{proof}

\begin{theorem}
    Assume $k< 3$. The solution given by Algorithm \ref{alg-no-magician-1d} is feasible. The expected approximation ratio of the solution given by Algorithm \ref{alg-no-magician-1d} is at least $\frac{1}{8}$.
\end{theorem}

\begin{proof}
The feasibility can be verified directly. Given the same realization of $Z_l^S$ in Step 3 of Algorithm \ref{alg-no-magician-1d}, the algorithm will temporarily open the same bins as those utilized by the procedure defined in Lemma \ref{lem:kleq3}.

Moreover, under this realization, item $p$ can only be assigned to bins where $Z_{l}^S=1$ and $p\in S$. This condition holds true for both Algorithm \ref{alg-no-magician-1d} and the procedure in Lemma \ref{lem:kleq3}.

Given the realizations of $Z_l^S$, the optimal reward for these \emph{temporarily opened} bins is achieved by retaining each item $p$ in the bin where it yields the highest reward $v_{lp}$, subject to $Z_{l}^S=1$ and $p\in S$. This is precisely what Step 10 of Algorithm \ref{alg-no-magician-1d} does for those \emph{temporarily} opened bins. Consequently, the total reward from these temporarily opened bins is at least as large as that derived from the procedure in Lemma \ref{lem:kleq3}.

The key difference from Theorem \ref{lem:kgeq3} is that the temporarily opened bins might violate the budget constraint. To address this, we compare the reward of the temporarily opened bin with the largest index ($l^*$ in Step 10 of Algorithm \ref{alg-no-magician-1d}) to the reward of all other temporarily opened bins ($\mathcal{L}^*\backslash\{l^*\}$ in Step 10 of Algorithm \ref{alg-no-magician-1d}). Whichever is larger will retain at least half of the total reward.

Moreover, based on point (b) of Lemma \ref{lem:kleq3}, either choice will stay within the budget. This comparison and selection process is implemented in Step 10 of Algorithm \ref{alg-no-magician-1d}.

Therefore, the expected reward of Algorithm \ref{alg-no-magician-1d} will be at least half of that derived from the procedure defined in Lemma \ref{lem:kleq3}, which yields the expected approximation ratio as $\frac{1}{2}\cdot \frac{1}{4} = \frac{1}{8}$.
\end{proof}

% \begin{theorem}\label{thm:no-magician}
%  The solution given by Algorithm \ref{alg-no-magician-1d} is feasible. The expected value of the solution given by Algorithm \ref{alg-no-magician-1d} is at least $\frac{1}{2}\max\left\{\frac{1}{4},\left(\frac{k-1}{k}\right)\cdot\left(1-\frac{1}{\sqrt{k}}\right)\right\}$ of the optimal solution.
% \end{theorem}

% \begin{proof}
%     The proof is delayed to Appendix \ref{appendix:alg-proof}.
% \end{proof}

\section{Applications in transportation}\label{sec:applications}
% In this section, we explain in detail how our proposed framework can be efficiently applied to solve GBCLPP and the generalized team orienteering problems.

% \subsection{Budeget-Constrained Line Planning Problem}\label{sec:grlpp} 

In this section, we introduce an application of GBAP in a multi-modal Mobility-on-Demand (MoD) setting, illustrating how the framework can be adapted to address practical transportation problems. Consider a MoD operator managing a fleet of single-capacity cars and high-capacity buses. The operator aims to provide a multi-modal service with buses operating on fixed routes and cars functioning in a demand-responsive manner. An individual trip request can be served by a bus, a car, or a combination of both.

This scenario could be envisioned as a transit agency planning a multi-modal service where traditional buses are complemented by a demand-responsive service, or a private rideshare operator integrating fixed-route shuttles with taxis. This problem was investigated by \cite{perivier2021real}. There is a wide range of literature studying the traditional line-planning problem. For instance, \cite{borndorfer2007column} introduced a new multicommodity flow model for line planning. Their model aims to strike a balance between minimizing operating costs for transport companies and reducing travel times for passengers. On the other hand, \cite{bertsimas2021data} provides a holistic approach for transit line planning using column generation. For additional literature on line planning problems in public transportation, readers are directed to \cite{schobel2012line}.

By applying GBAP to this MoD scenario, we can optimize resource allocation and route planning, maximizing the efficiency and service quality of the transportation system while adhering to budget constraints. This demonstrates the versatility and practical value of GBAP in addressing complex, real-world transportation challenges.

In this problem, the transit network is modeled as an undirected weighted graph $G=(V,E)$, where $V$ contains $n$ nodes as potential origin/destination nodes, and $E$ is the set of edges representing the shortest paths between the nodes. The weights of the edges are the cost required to cross an edge $e \in E$. A set $\mathcal{L}$ of candidate bus routes/lines on $G$ is given to the operator. Each candidate route / line is defined as a fixed sequence of consecutive edges of $G$, and it has an operating cost $c_l$. There is a budget $B$ for operating bus lines. Each bus has capacity $C$ and each line $l$ has frequency $f_l$. Thus, the overall capacity of a bus line is $C\cdot f_l$. The cars cover 'first-last mile' travel. It is assumed that there are no inter-bus transfers for each passenger. If the passenger $p$ is assigned to line $l$, then $v_{lp}\geq 0$ is defined as the welfare value contributed to the system by this matching. The operator is responsible for selecting a set of bus lines from $\mathcal{L}$ to operate in a way that maximizes the overall welfare of the system. The value of $v_{lp}$ is calculated according to the rules given, e.g. binary based on whether the matching is feasible or the number of car miles saved due to the matching (for more details, please see \cite{perivier2021real}). The goal is to find a subset of lines to be operated and an assignment of passengers to opened lines to maximize the welfare of the system, such that: 
    (a) the cost of opened lines does not exceed the budget;
    (b) the assignment of passengers and their trips to the opened line $l$ respect the capacity constraints of each edge on line $l$; 
    (c) a passenger is assigned to at most one line.

Now we can describe how this problem fits into the proposed framework. For each line $l$ with $n_l$ edges, we map it to a bin in GBAP with an $n_l$-dimensional capacity vector. Each entry in the capacity vector of line $l$ is set to $C\cdot f_l$. Each passenger is taken as an item in GBAP with a $0-1$ weight vector $r_{lp}:=(r_{lp}^{(1)},\ldots,r_{lp}^{(n_l)})\in\{0,1\}^{n_l}$. If the $i$th edge of line $l$ is used when passenger $p$ is matched to line $l$, then we set the $i$th entry of $r_{lp}$ to $1$, and $0$ otherwise. Note that $r_{lp}$ has consecutive $1$s according to the definition of a route in RLPP. The cost of utilizing bin $l$ is set to be $c_l$ while the cost of assigning item $p$ to bin $l$ is assumed to be $0$. Contributed welfare values $\{v_{lp}\}$ are the coefficients in the objective of (\ref{eq:2D-IP-BAP-UW}). Since each passenger can only be matched to at most one line, we let $\rho_p=1$ for each $p\in[P]$. {The RLPP is an instance of GBAP with an additional \emph{aggregation} step that merges any subset of selected lines with the same route (and different frequencies) to a single line with the corresponding aggregated frequency.} Since the aggregation step neither increases the cost of the solution nor decreases the objective values as shown in \cite{perivier2021real}, and can be trivially incorporated into our framework, we omit it here for simplicity.

P\'erivier et al. \cite{perivier2021real} present an approximation algorithm (see Algorithm \ref{alg:rlpp} in Appendix \ref{appendix:comparison}) to solve the problem based on the randomized rounding scheme proposed by \cite{fleischer2011tight}. The corresponding approximation ratio is $1-\frac{1}{e}-\epsilon$, with a probability of violating the budget constraint bounded by $e^{-\epsilon^2k/3}$, where $k:=\frac{B}{\max_{l\in \mathcal{L}}c_l}$ as we have assumed before.

\begin{remark}\label{remark:tradeoff}
    We note that the violation probability bound, given by $e^{-\frac{k}{3}\epsilon^2}$, may surpass $\frac{1}{2}$, even when $k$ is as large as $\frac{1}{\epsilon^2}$. To address this, one may consider setting $\epsilon$ to $\frac{1}{k^\gamma}$, with $\gamma < \frac{1}{2}$. However, this implies that the method's approximation ratio becomes $\left(1-\frac{1}{e}\right)\left(1-\frac{1}{k^\gamma}\right)<\left(1-\frac{1}{e}\right)\left(1-\frac{1}{\sqrt{k}}\right)$ with violation probability bound $e^{-\frac{1}{3}k^{1-2\gamma}}$. We present Table \ref{table:comparison} to illustrate this by taking $k=25,50,100$, which are reasonable and relatively large parameters in real-world cases. Note that by definition, the number of the opened lines in the optimal solution can be significantly larger than $k$. We calculate the approximation ratio as $\left(1-\frac{1}{e}\right)\left(1-\epsilon\right)$ since it is the original one derived by \cite{perivier2021real} and is higher than $\left(1-\frac{1}{e}-\epsilon\right)$.  Furthermore, note that the rounding process tends to produce infeasible solutions with higher objective values than feasible ones, suggesting that the actual expected approximation ratio for feasible solutions might be significantly lower than $\left(1-\frac{1}{e}\right)\left(1-\epsilon\right)$, especially when the violation probability is non-negligible. 
   
    %However, when $k=50$ and $\gamma = \frac{1}{3}$, for instance, the violation probability bound surpasses $0.29$. Furthermore, the rounding process tends to produce infeasible solutions with higher objective values than feasible ones, suggesting that the actual expected approximation ratio for feasible solutions might be considerably lower than $\left(1-\frac{1}{e}-\epsilon\right)$ in this case. When $k=50$ and $\gamma=\frac{1}{4}$, the approximation ratio declines to less than $0.4$, and the violation probability bound is approximately $0.095$. In contrast, Algorithm \ref{alg-1d-better} and Algorithm \ref{alg-no-magician-1d} achieve an approximation ratio greater than $0.4$ in this case by setting the $\epsilon$ term to be a sufficiently small constant. %Note that the $\epsilon$ term in the approximation ratio for Algorithm \ref{alg-1d-better} and Algorithm \ref{alg-no-magician-1d} are only influence  the accuracy parameter for stopping rules in ellipsoid methods

    \begin{table}[ht]
\centering
\begin{tabular}{|c|c|c|c|c|}
\hline
$k$ & $\gamma$ & $\epsilon$ & $\left(1-\frac{1}{e}\right)\left(1-\epsilon\right)$ & $e^{-\frac{k}{3}\epsilon^2}$ \\ \hline
25  & 1/3       & 0.3420     & 0.4159                & 0.3773                         \\ \hline
25  & 1/4       & 0.4472     & 0.3494                & 0.1889                         \\ \hline
50  & 1/3       & 0.2714     & 0.4605                & 0.2929                         \\ \hline
50  & 1/4       & 0.3761     & 0.3944                & 0.0947                         \\ \hline
100 & 1/3       & 0.2154     & 0.4959                & 0.2128                         \\ \hline
100 & 1/4       & 0.3162     & 0.4322                & 0.0357                         \\ \hline
\end{tabular}
\caption{Results for the given expressions with varying $k$ and $\gamma$ values.  Algorithm \ref{alg-no-magician-1d} achieve approximation ratios greater than $0.39$, $0.4$, and $0.44$ for $k=25, 50, 100$, respectively, by setting the $\epsilon$ term to be a sufficiently small constant. }\label{table:comparison}
\label{tab:results}
\end{table}

We note that when  Algorithm \ref{alg-no-magician-1d}, and Algorithm \ref{alg:rlpp} (in Appendix \ref{appendix:comparison}) utilize the same LP solution for their rounding process,  Algorithm \ref{alg-no-magician-1d} are provably superior to Algorithm \ref{alg:rlpp}. The reason for this is that the feasible solutions produced by Algorithm \ref{alg:rlpp} can also be generated by  Algorithm \ref{alg-no-magician-1d} with the same probability distributions, while each infeasible solution generated by Algorithm \ref{alg:rlpp} is converted into a feasible one by Algorithm \ref{alg-no-magician-1d}. However, since the approximation ratio derived for Algorithm \ref{alg:rlpp} encompasses the objective values of infeasible solutions, it is not applicable to  Algorithm \ref{alg-no-magician-1d}. 
 \end{remark}
%\SSC{Shouldn't the remark end here?}
 
%Therefore, if we apply Algorithm \ref{alg-1d-better} and Algorithm \ref{alg-no-magician-1d}, then the corresponding approximation ratio is $\frac{1}{2}\left(1-\frac{1}{\sqrt{k}}\right)-\epsilon$ and $\frac{1}{2}\left(\frac{k-1}{k}\right)\left(1-\frac{1}{\sqrt{k}}\right)$ respectively. We also present a hardness result for RLPP showing that the approximation ratio cannot be better than $(1-\frac{1}{e})$ in Appendix \ref{appendix:hard}.

% \begin{theorem}\label{hardness-RLPP}
% For any $\epsilon>0$ and positive integer $U$, there is always an integer $k$ such that $k\geq U$, and RLPP with $k=\frac{B}{\max_lc_l}$ cannot be approximated in polynomial time within a ratio of $(1-\frac{1}{e}+\epsilon)$ even when $v_{l_1p}=v_{l_2p}$ for $l_1,l_2\in[L]$ and each $p\in[P]$ unless $P=NP$.
% \end{theorem}

% \begin{remark}
% We present these results in the following theorem. For a more comprehensive discussion, please refer to Appendix \ref{appendix:comparison}.
% \end{remark}

% {\begin{theorem}
% For the method in \cite{perivier2021real} (see Algorithm \ref{alg:rlpp} in Appendix \ref{appendix:comparison}), if the objective value is set to be $0$ when it returns an infeasible solution, then with the same realizations of $S_l$ for each $l\in[L]$, the objective values of the solutions found by  Algorithm \ref{alg-1d-better} and Algorithm \ref{alg:modified-no-magician} {(modified from Algorithm \ref{alg-no-magician-1d}, see Appendix \ref{appendix:comparison})} are higher than that from Algorithm \ref{alg:rlpp}.
% \end{theorem}}

\subsection{Numerical Experiments.}

In this section, we conduct numerical experiments in the context of RLPP based on the data provided by \cite{perivier2021real} and Algorithm \ref{alg-no-magician-1d}. We also compare its performance with Algorithm \ref{alg:rlpp} with $\epsilon = 0.05$. The underlying road network is derived from OpenStreetMap (OSM) geographical data from \cite{boeing2017osmnx}. The size of the candidate set of lines is set to be $1000$, and the lines are generated based on the method proposed by \cite{silman1974planning}. The passenger data comes from records of for-hire vehicle trips in Manhattan using the New York City Open Data operator, with a time window between 5 pm and 6 pm on the first Tuesday of April 2018. There are $13,851$ trip requests. The bus capacity is set to be $30$. For practical purposes, we solve the corresponding LP via column generation and apply a timeout—the current LP solution will be returned once the time limit is exceeded.

We implement Algorithm \ref{alg-no-magician-1d}, Algorithm \ref{alg:rlpp}, and the modified Algorithm \ref{alg-no-magician-1d} with rounding based on the same LP solution used by Algorithm \ref{alg:rlpp}. The experiments consider four different budgets, and we simulate the rounding $10^4$ times. In Figure \ref{fig:comparison}, the $x$-axis represents the number of simulations, and the $y$-axis indicates the objective value of the best solution found by the algorithms so far. We omit the first $100$ realizations in Figure \ref{fig:comparison} to make the display of the $y$-axis more readable. The results\footnote{In the dataset we employed, the value of \( k \) substantially exceeds \( 3 \). In this case, Algorithm \ref{alg-no-magician-1d} coincides with Algorithm 2 presented in our conference paper \cite{jiang2022approximation}, yielding same numerical outcomes.
} show that when the budget is $8\times 10^4$ and $10^5$ respectively, our algorithms find significantly better solutions, while the three methods perform similarly when the budget is $2\times 10^4$ and $5\times 10^4$ respectively.

\begin{figure}%[H]
\centering
\includegraphics[width=12cm]{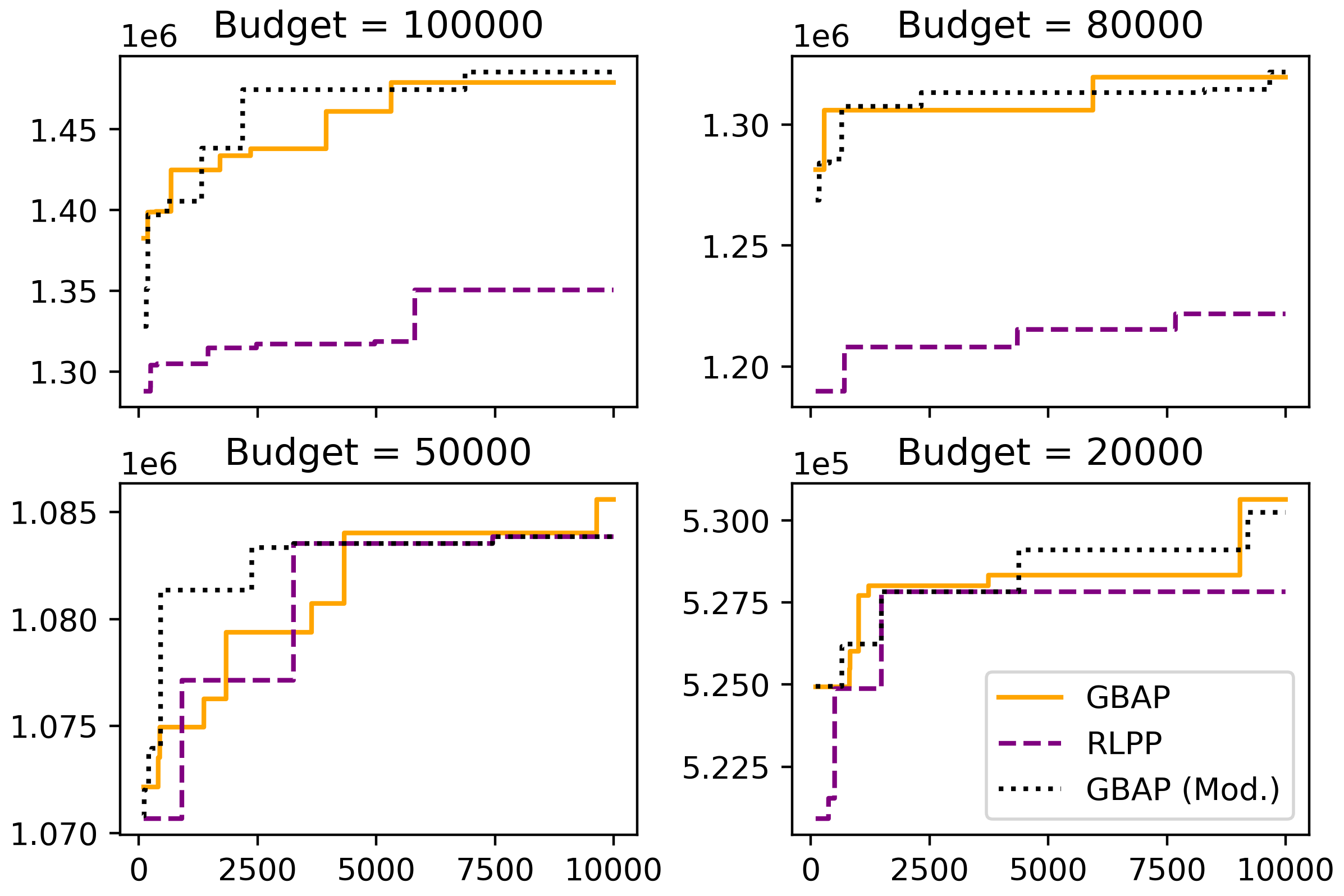}
\caption{We compared three methods: Algorithm \ref{alg-no-magician-1d} (GBAP), Algorithm \ref{alg:rlpp} (RLPP) (with $\epsilon=0.05$), and Modified Algorithm \ref{alg-no-magician-1d} (which uses the same LP solution for rounding as Algorithm \ref{alg:rlpp}). Note that Modified Algorithm \ref{alg-no-magician-1d} always performs better than Algorithm \ref{alg:rlpp} since they use the same realizations produced by the rounding process, which is based on the same LP solution.\label{fig:comparison}
}
\end{figure}

\section{Conclusion and Future Directions}\label{sec:conclusion}

In this paper, we have introduced a novel approach to the Generalized Budgeted Assignment Problem (GBAP) and demonstrated its application to the Real-Time Line Planning Problem (RLPP) in transportation systems. To the best of our knowledge, this is the first algorithm to achieve a constant approximation ratio for GBAP. Our approach not only enhances the theoretical understanding of GBAP but also provides practical solutions for complex transportation planning challenges.

Our contributions are threefold. First, we have developed a randomized approximation algorithm for GBAP with a proven constant approximation ratio. Second, we have introduced an innovative technique that employs randomized rounding realizations to simulate two distinct online mechanisms. This approach enables the simultaneous satisfaction of different constraints while maintaining a strong theoretical guarantee on algorithm performance. Third, we have applied our algorithm to the RLPP and demonstrated its effectiveness through numerical experiments using real-world data from New York City's transportation system. This application underscores the algorithm's potential for large-scale urban transit planning and bridges the gap between theoretical advancements and practical implementations in transportation systems.

While our results represent a significant step forward, several avenues for future research remain open.

As stated in out extended abstract at COCOON \cite[Theorem 6]{jiang2022approximation}, we have established that the approximation ratio cannot exceed $1-\frac{1}{e}$ by reducing the problem to the max $k$-cover problem.\footnote{We have chosen to omit the detailed statement of this result to maintain focus on our primary contributions.} An avenue for future research is to explore whether the upper bound can be tightened to a constant strictly less than $1-\frac{1}{e}$.

The second potential direction for future research is to explore possible improvements to our current approximation ratio. While our algorithm (Algorithm \ref{alg-no-magician-1d}) enhances practical efficiency and performance compared to the two alternative procedures presented in Lemma \ref{lem:kgeq3} and Lemma \ref{lem:kleq3}, its approximation ratio guarantee remains the same as these procedures since it essentially builds upon their bounds. It remains unclear whether there is a significant gap in the theoretical guarantees between our algorithm and these two alternative procedures.
 Therefore, refining the analysis to enhance the theoretical bound of Algorithm \ref{alg-no-magician-1d} may be worthwhile. Another avenue could be the development of new algorithms with a better approximation ratio guarantee. However, it is important to note that our primary goal has been to develop an algorithm that is both theoretically sound and practically efficient. This dual requirement presents significant challenges in substantially improving the approximation ratio without compromising computational efficiency. Any future improvements will need to carefully balance these competing objectives, making this a complex but potentially rewarding area for further investigation.

It is also interesting for future research to explore scenarios with stochastically appearing items, each following a known binary distribution. In this context, the objective would be to maximize expected rewards by strategically utilizing bins within the given budget constraint. This extension could enhance the model's applicability to real-world situations where demand or item availability is uncertain but follows predictable patterns.

\bibliographystyle{plainnat}
\bibliography{bib_assignment}

\begin{appendices}

\section{Proof of Lemma \ref{lem:kleq3}}\label{appendix:proof-of-lem-kleq}

\begin{proof}[Proof of Lemma \ref{lem:kleq3}] 
Point (a) can be proven using the exact same way as in Lemma \ref{lem:kleq3}. Regarding point (b) and (c), they can be directly verified from its definition.

Following the same reasoning as in \eqref{eq:xi=H}, 
 the expected reward derived from the procedure in Lemma \ref{lem:kleq3} can be expressed as:

\begin{align}
    \E\left[\sum_{l\in[L]}\sum_{p\in [P]}v_{lp}Z_l^p\cdot \widetilde{Y}_{lp}\cdot \xi_l\right]
    &=\E\left[\sum_{l\in[L]}\sum_{S\in I_l}\sum_{p\in S}v_{lp}Z_l^pZ_l^S\cdot \xi_l\right]\nonumber\\
    &\geq  \frac{1}{2}\sum_{l\in[L]}\E\left[\sum_{S\in I_l}\sum_{p\in S}v_{lp}Z_l^S\Bigg\vert H_l = 1\right]\cdot P\left[\xi_l=1\right]\label{eq:xi=H-2}
\end{align}

Same as in Lemma \ref{lem:kleq3}, let us define 
\[ r_l := \E\left[\sum_{S \in I_l}\sum_{p \in S} v^S_{lp} Z_l^S \mid H_l = 1\right] = \frac{\sum_{S \in I_l}\left(\widetilde{X}_{lS} \sum_{p \in S} v_{lp}\right)}{\sum_{S \in I_l} \widetilde{X}_{lS}}. \] 
Subsequently, Eq.~\eqref{eq:xi=H-2} can be expressed as

\begin{align}\label{eq:reduce-online-knapsack-2}
    \frac{1}{2}\sum_{l \in [L]} \E\left[r_l \cdot \xi_l\right] = \frac{1}{2} \E\left[\sum_{l \in [L]} r_l \cdot \xi_l\right].
\end{align}

To establish the lower bound of Eq.~(\ref{eq:xi=H-2}), we again construct an instance of the online fractional knapsack problem, as in Definition~\ref{def:online-knapsack}, consisting of \( L \) rounds. In each round \( l \), an object appears with a probability \( P(H_l = 1) \), possessing a weight \( c_l \) and a reward \( r_l \). The initial capacity is equal to the budget $k$. As specified in Step 2 of Algorithm~\ref{alg-no-magician-1d}, it satisfies that \( \frac{r_1}{c_1} \geq \frac{r_2}{c_2} \geq \ldots \geq \frac{r_L}{c_L} \), allowing us to apply Lemma~\ref{expected-reward-knapsack}.

Following the same reasoning as in Lemma \ref{lem:kleq3}, expected rewards of this online knapsack problem instance are at least:

\begin{align*}
    \frac{1}{2}\sum_{l\in [L]}r_l\cdot P(H_l=1)=\frac{1}{2}\sum_{l\in[L]}\sum_{p\in [P]}\sum_{S\in I_l:p\in S}v_{lp}\widehat{X}_{lS}.
\end{align*}

Unlike Lemma \ref{lem:kgeq3}, the definition of \(\xi_l\) provided in Lemma \ref{lem:kleq3} indicates that the reward cannot exceed \(\sum_{l\in [L]}r_l\cdot \xi_l\). Consequently, we establish:

\begin{align}
    \E\left[\sum_{l\in [L]}r_l\cdot \xi_l\right] \geq \frac{1}{2}\sum_{l\in[L]}\sum_{p\in [P]}\sum_{S\in I_l:p\in S}v_{lp}\widehat{X}_{lS}.\label{eq:keep-last-bin}
\end{align}

Summarizing \eqref{eq:xi=H-2},\eqref{eq:reduce-online-knapsack-2} and \eqref{eq:keep-last-bin}, we conclude that

$$\E\left[\sum_{l\in[L]}\sum_{S\in I_l}\sum_{p\in S}v_{lp}Z_l^pZ_l^S\cdot \xi_l\right] \geq \frac{1}{2}\E\left[\sum_{l\in [L]}r_l\cdot \xi_l\right] \geq \frac{1}{4}\sum_{l\in[L]}\sum_{p\in [P]}\sum_{S\in I_l:p\in S}v_{lp}\widehat{X}_{lS}.$$

\noindent This completes the proof.
\end{proof}

\section{ Method from \cite{perivier2021real}}\label{appendix:comparison}

Here we restate the method from \cite{perivier2021real} while omitting the aggregation step.

\begin{alg}\label{alg:rlpp}
~

\begin{enumerate}
        \item Let $\widehat{X}$ be the solution to the LP relaxation of (\ref{eq:2d-IP-config-original}).
    \item\label{alg-rlpp-rounding} Independently for each $l\in[L]$, 
    randomly select a set $T_l$ from $I_l\cup\{\emptyset\}$ following the distribution below:
   \begin{equation*}
        P({T}_l=S)=
        \begin{cases}
      1-\sum_{S\in I_l}\widehat{X}_{lS} & \text{if }S=\emptyset\\
      \widehat{X}_{lS} & \text{if }S\in I_l
    \end{cases} .
    \end{equation*}

Let the (random) decision  variable be $Z_l^S$, i.e., $Z_l^S=1$ if $S=T_l$, and $Z_l^S=0$ otherwise for $S\in I_l\cup\{\emptyset\}$.
 \item Temporarily assign all the elements in $T_l$ to bin $l$ for each $l\in [L]$.
    \item For $p\in [P]$, we remove each item $p$ that is assigned to more than $\rho$ bins from all but the $\rho$ bins with highest rewards $v_{lp}$ it is assigned to.
    \item Open all the lines that have been assigned with at least one passenger, and close all the other lines. If this assignment does not exceed the budget, then we take this assignment as the final solution.
\end{enumerate}
\end{alg}

\end{appendices}

\end{document}